\documentclass[twocolumn]{autart}    % Enable this line and disable the 
\usepackage{graphicx}
\usepackage{amssymb,mathrsfs}
\usepackage{color}
\usepackage{xcolor}
\usepackage{graphicx}
\usepackage{comment}
\usepackage{float}
\usepackage{mathtools}
\usepackage{framed}

\mathtoolsset{showonlyrefs}
\usepackage{lipsum}
\usepackage{amsfonts}
\usepackage{graphicx}
\usepackage{epstopdf}
\usepackage{algorithmic}
\usepackage[normalem]{ulem}
\usepackage{enumitem}

\newcommand{\R}{\mathbb{R}}

\newcommand{\F}{\mathscr{F}}
\newcommand{\E}{\mathbb{E}}

\newcommand{\ud}{\,\mathrm{d}}
\newcommand{\QV}[2]{\left\langle #1, #2\right\rangle}

\newcommand{\BAR}[1]{\overline{#1}}

\newcommand{\transpose}{^{\operatorname{T}}}
\newcommand{\mc}[1]{\mathcal{#1}}

\newcommand{\PP}{\mathbb{P}}

\newcommand{\mb}[1]{\mathbf{#1}}

\newcommand{\Tr}{\mathrm{Tr}}

\begin{document}

\begin{frontmatter}

\title{Optimal Design of Stealthy Attacks in Partially Observed Linear Systems: A Likelihood-Based Approach\thanksref{footnoteinfo}}

\thanks[footnoteinfo]{The material in this paper was not presented at any conference.
R.H. was partially supported by the ONR grant under \#N00014-24-1-2432, the Simons Foundation (MP-TSM-00002783) and the NSF grant DMS-2420988.}

\author[1]{Haosheng Zhou}\ead{hzhou593@ucsb.edu}, \author[1,2]{Ruimeng Hu\corauthref{corr}}\ead{rhu@ucsb.edu}          
\corauth[corr]{Corresponding author.}

\address[1]{Department of Statistics and Applied Probability, University of California, Santa Barbara, CA 93106-3110, USA.} 
\address[2]{Department of Mathematics, University of California, Santa Barbara, CA 93106-3080, USA.}


\begin{abstract}      
We study the optimal design of stealthy attacks against partially observed linear control systems. We first propose a novel likelihood-based detection mechanism derived from the innovation process, based on which we quantify stealthiness and formulate an attack design problem that trades off performance degradation and detectability. We develop a tractable control-theoretic framework for optimal stealthy attacks under two information structures: deterministic attacks fixed prior to system evolution, and adaptive attacks constructed from available observations. In the adaptive setting, the attacker’s partial observation leads to a stochastic control problem with an endogenous information structure. We address this challenge through a hierarchical optimization framework combined with the separation principle, reducing the problem to a Markovian control formulation and yielding semi-explicit optimal attacks. We further establish well-posedness of the resulting systems and illustrate through numerical experiments how information constraints shape the trade-off between attack effectiveness and stealthiness.
\end{abstract}

\end{frontmatter}

\section{Introduction}
Since the early development of control systems, ensuring resilience against adversarial disturbances has been a fundamental concern. 
While sufficiently large disturbances can, in principle, degrade system performance arbitrarily, such attacks are typically unrealistic in practice, as large anomalies can be detected by monitoring mechanisms. Once detected, the system operator may respond by modifying the control policy or shutting down the system, thereby limiting the attacker’s impact. This observation highlights a key constraint in adversarial design: effective attacks must remain difficult to detect.

Motivated by this challenge, the notion of \emph{stealthy attacks} has emerged as a natural modeling framework, in which the attacker seeks to degrade system performance while remaining undetected by the monitoring mechanisms deployed by the agent. 
The optimal design of stealthy attacks has recently attracted increasing attention, particularly in the context of the security of cyber–physical systems \cite{kwon2013security,sui2020vulnerability}, critical infrastructures \cite{lehto2022cyber}, and aerial systems \cite{kwon2014analysis}.

The detection of malicious attacks in control and estimation systems has been extensively studied in the literature. A prominent class of model-based methods relies on the innovation process, such as the classical \(\chi^2\)-detector \cite{willsky2003generalized} which exploits its Gaussian structure. Other approaches detect inconsistencies between estimated and observed signals \cite{jiang2025attack,rawat2015detection} or employ likelihood-based criteria \cite{das2024almost}. In parallel, data-driven methods based on machine learning and statistical techniques have also been developed, including support vector machines and neural networks for attack detection \cite{chen2005application}, graph-based autoregressive models \cite{sharma2013fault}, and dynamic mode decomposition \cite{mobini2024online}.
We refer interested readers to \cite{tan2020brief,zhang2021survey} for  comprehensive overviews.

Complementary to attack detection, the design of stealthy attacks has been investigated under various system models and detection schemes. Existing works typically focus on discrete-time systems and enforce stealthiness through hard constraints on test statistics, such as $\chi^2$-statistics \cite{chen2017optimal,mo2010false}, hypothesis-testing-based criteria \cite{fang2019stealthy,ren2021kullback}, or innovation-based criteria in which stealthiness is measured by the Kullback-Leibler divergence  \cite{jin2022optimal}. The work \cite{sun2025secure} is one of the few studies of attacked dynamics in a continuous-time setting; however, it focuses on secure state estimation rather than attack design.
Interested readers are referred to the survey \cite{lian2024survey} for recent developments in stealthy attack design. 
While these approaches provide valuable formulations of the trade-off between performance degradation and detectability, they are often developed from a primarily numerical perspective and do not yield analytical or semi-explicit characterizations of optimal attacks. Moreover, continuous-time formulations and adaptive attack strategies under partial observation remain relatively underexplored.

In this paper, we develop a continuous-time likelihood-based stealthiness framework for partially observed linear control systems, which exploits the linear-quadratic (LQ) structure of the resulting attack-design problem and provides tractable characterizations of optimal stealthy attacks under both deterministic and adaptive information structures.
The system consists of an agent and an attacker, both with full knowledge of the system dynamics but operating under partial observation.
The agent aims to accomplish a primary task modeled as a stochastic control problem but remains unaware of the presence of attacks, whereas the attacker injects disturbances into the state-observation dynamics to degrade the agent's performance.
Anticipating that the agent may employ a likelihood-based detector, the attacker incorporates a notion of stealthiness as a precautionary consideration, balancing attack effectiveness against detectability under different information structures. The main contributions of this paper are as follows:
\begin{enumerate}
\item We propose a likelihood-based attack detector via a Girsanov-type argument, which provides  a tractable path-space criterion for stealthiness.
In contrast to classical innovation-based detectors such as the \(\chi^2\)-detector \cite{willsky1976survey}, the proposed approach exploits the full distributional structure of the innovation process, leading to a more informative measure of detectability.

\item We develop a tractable framework for stealthy attack design in continuous time. In the deterministic setting, the problem reduces to a LQ control problem that admits semi-explicit solutions characterized by coupled systems of Riccati equations, together with well-posedness guarantees. In the adaptive setting, the attacker operates under partial observation, which leads to a stochastic control problem with control-dependent innovations.
We address this challenge by combining a hierarchical optimization approach with the separation principle, yielding semi-explicit characterizations of optimal adaptive attacks. To the best of our knowledge, this is one of the first
continuous-time formulations that integrates likelihood-based stealthiness with adaptive attack design under
partial observation.

\item Numerical experiments demonstrate that the proposed designs achieve improved trade-offs between effectiveness and stealthiness compared to heuristic approaches, illustrating the impact of information constraints on optimal attack strategies.

\end{enumerate}

The rest of the paper is organized as follows: 
Section~\ref{sec:agent} formulates the agent's primary task as a partially observable control problem and quantifies the performance degradation under attacks.
Section~\ref{sec:detector} derives the likelihood-based attack detector and introduces the notion of stealthiness. 
Sections~\ref{sec:det} and~\ref{sec:adpt} study the design of deterministic and adaptive stealthy attacks under different information structures. 
Numerical results are presented in Section~\ref{sec:numerics}, and concluding remarks appear in Section~\ref{sec:conclusions}.
For presentational clarity, all proofs are deferred to the supplementary material.

\textbf{Notations.}
Fix a finite horizon \([0,T]\). 
For a stochastic process \(\{X_t\}\), let \(\{\F^X_t\}\) be its natural filtration, i.e., \(\F^X_t := \sigma(X_s,\ \forall s\in[0,t])\).
Denote by \(X_{[0,t]} := \{X_s\}_{s\in[0,t]}\) the path of \(X\) up to time \(t\), and by \(\mu_X\) its law on the path space.
The space of square-integrable \(\R^d\)-valued random variables is denoted by \(L^2(\R^d)\).
For subsets \(U,V\) of finite-dimensional Euclidean spaces, let \(C(U;V)\) be the space of continuous maps from \(U\) to \(V\).
We write \(\mathrm{Concat}(\cdot)\) for vertical concatenation of vectors or matrices.
Let \(\mathbb{S}^{d\times d}\) be the set of symmetric \(d\times d\) matrices, and \(\mathbb{S}^{d\times d}_+\) (resp. \(\mathbb{S}^{d\times d}_{++}\)) the subset of positive-semidefinite (resp. positive-definite) matrices.
Denote by \(I_d\) (resp. \(0_{d_1\times d_2}\)) the \(d\times d\) identity matrix (resp. \(d_1\times d_2\) zero matrix); we write \(I\) and \(0\) when dimensions are clear.
Throughout,  \(\|\cdot\|\) denotes the Euclidean norm for vectors and the matrix \(2\)-norm for matrices. 
Matrix inequalities are in the positive-semidefinite sense. 
For a vector/matrix-valued function \(A:[0,T]\to \R^{d_1\times d_2}\), define \(\|A\| := \sup_{t\in[0,T]}\|A_t\|\). 
Finally, let \(x\vee y := \max\{x,y\}\) and \([N] := \{1,2,\ldots,N\}\) for \(N\in\mathbb{N}_+\).

\section{Agent's optimal control and attacked dynamics under partial observability}\label{sec:agent}
This section introduces the agent’s baseline control problem under partial observation and characterizes its optimal solution. We then incorporate adversarial perturbations into the state and observation dynamics and derive the resulting closed-loop system under the agent’s (misspecified) optimal response, which provides the basis for quantifying performance degradation under attacks.

\subsection{Model setup}
Let \((\Omega,\F,\{\F_t\}_{t\geq 0},\PP)\) be a filtered probability space supporting two independent Brownian motions \(V\) and \(W\) in \(\R^p\) and \(\R^q\), respectively, where \(\F_t := \sigma(V_s,W_s,\ \forall s\in[0,t])\) is the natural filtration.
The agent considers a partially observed linear system, in which the (unobserved) state process \(X \in \R^d\) and observation process \(Y \in \R^m\) evolve as:
\begin{align}
    \label{eqn:X}
    \ud X_t &= (A_tX_t + B_tu_t + a_t)\ud t + \sigma_V\ud V_t,\\
    \label{eqn:Y}
    \ud Y_t &= (H_tX_t + h_t)\ud t + \sigma_W\ud W_t,
\end{align}
with initial condition \(X_0\sim N(x_0,R_0)\) and \(Y_0 = 0\), where \(u \in \R^c\) is the control process, and \(V\) and \(W\) represent exogenous state and observation noises, respectively. The coefficients $(A, B, H, a, h)$ are deterministic, continuous, and of compatible dimensions. We further assume that \(d\leq p\), \(m\leq q\), \(\mathrm{rank}(\sigma_V) = d\), and \(\mathrm{rank}(\sigma_W) = m\), so that \(\sigma_V\sigma_V\transpose\) and \(\sigma_W\sigma_W\transpose\) are positive definite.

The agent chooses \(u\) to minimize the expected cost
\begin{equation}
    \label{eqn:J_agent}
    J(u) := \E \int_0^T [(X_t - r_t)\transpose Q_t (X_t - r_t) + u_t\transpose S_t u_t]\ud t ,
\end{equation}
where \(Q\in C([0,T];\mathbb{S}^{d\times d}_{+})\), \(S\in C([0,T];\mathbb{S}^{c\times c}_{++})\), and \(r\in C([0,T];\R^d)\).
This defines a linear-quadratic (LQ) tracking problem with the standard performance-energy trade-off: the agent seeks to track the reference trajectory \(r\) while penalizing excessive control effort.

Since the state $X$ is not directly observed, admissible controls must be adapted to the observation filtration, i.e., \(u_t\in \F^{Y}_t,\ \forall t\in[0,T]\). We consider the admissible class \(\mathscr{A}^a\) generated by finite-dimensional observation-driven features \(\iota\) (see \cite[Definition~5.3.1]{davis1977}), consisting of controls of the form:
\begin{equation}
    \label{eqn:adm_separation}
    u_t = K^\iota_t\iota_t + L^\iota_t,
    \quad \ud \iota_t = (\Gamma_t\iota_t + \gamma_t)\ud t + \delta_t\ud Y_t,
\end{equation}
where \(\iota_0 \in L^2(\R^{d_\iota})\) and the coefficients are measurable functions of compatible dimensions.
For any \(u\in\mathscr{A}^a\), the control \(u_t\) depends on the observation history \(Y_{[0,t]}\), reflecting the path-dependent nature of the partially observable control problem.

\subsection{Solution via the separation principle}\label{subsec:separation}
The main technical difficulty in solving \eqref{eqn:X}--\eqref{eqn:J_agent} lies in the coupling between control optimization and state estimation: the control \(u\) influences the unobserved state \(X\), which is only indirectly accessible through the observation \(Y\). 
Consequently, optimal decision-making must rely on filtered information, whose evolution is, in principle, control-dependent.

In the linear-Gaussian setting, this difficulty is resolved by the separation principle \cite[Section~5.3]{davis1977}, which allows filtering and control optimization to be treated separately: the optimal control \(u^*\in\mathscr{A}^a\) is Markovian in the filtered state, while the filtering covariance evolves independently of the control. 

To this end, define the filtered estimate \(\hat{X}_t := \E(X_t|\F^Y_t)\) and the whitened innovation process associated with equations~\eqref{eqn:X}--\eqref{eqn:Y}, 
$$\ud I_t := (\sigma_W\sigma_W\transpose)^{-\frac12}[\ud Y_t - (H_t\hat{X}_t + h_t)\ud t], \quad I_0 = 0.$$
By the separation principle \cite[Proposition~5.3.2]{davis1977} and the Kalman-Bucy filter \cite{bain2009fundamentals}, \(I\) is an \(\{\F^Y_t\}\)-adapted Brownian motion independent of the control \(u\), and \(\hat{X}\) satisfies
\begin{equation}
    \label{eqn:X_hat}
    \ud \hat{X}_t = (A_t \hat{X}_t + B_tu_t + a_t)\ud t + R_tH_t\transpose (\sigma_W\sigma_W\transpose)^{-\frac12}\ud I_t,
\end{equation}
with \(\hat{X}_0 = x_0\).
Here, the conditional covariance matrix \(R_t := \mathrm{cov}(X_t|\F^Y_t)\) solves the Riccati ODE
\begin{equation}
    \label{eqn:R}
    \dot{R}_t = A_tR_t + R_tA_t\transpose+\sigma_V\sigma_V\transpose - R_tH_t\transpose(\sigma_W\sigma_W\transpose)^{-1}H_tR_t,
\end{equation}
with initial condition \(R_0\).
The filtering equation~\eqref{eqn:X_hat} is linear, and the Riccati equation~\eqref{eqn:R} is globally well-posed by standard Riccati equation theory \cite[Theorem~4.1.6]{abou2012matrix}. 

Using the filtered state $\hat X_t$, the cost~\eqref{eqn:J_agent} decomposes as $$J(u) = \E \int_0^T [(\hat{X}_t - r_t)\transpose Q_t (\hat{X}_t - r_t) + u_t\transpose S_t u_t]\ud t  + C,$$ where \(C\) is independent of the control \(u\) (see \cite[equation~(5.50)]{davis1977}). Consequently, the partially observable control problem reduces to a fully observable Markovian LQ control problem in the filtered state \(\hat{X}\), whose dynamics are driven by the $\{\F^Y_t\}$-adapted Brownian motion \(I\). The optimal control in \(\mathscr{A}^a\) then takes the feedback form \(u_t = \phi^u(t,\hat{X}_t)\), with a linear structure specified below.

\begin{prop}
    \label{prop:solution_agent}
    For the agent's optimization~\eqref{eqn:X}--\eqref{eqn:J_agent}, the optimal control at time $t$, given  \(\hat{X}_t = x\), is
    \begin{equation}
    \label{eqn:opt_ctrl}
    u^*(t,x) = -S_t^{-1}B_t\transpose F_t x - \tfrac12 S_t^{-1}B_t\transpose \mb{f}_t,
    \end{equation}
    where \(F\in C([0,T]; \mathbb{S}^{d\times d})\) and \(\mb{f}\in C([0,T]; \R^d)\) solve
    \begin{equation}
    \label{eqn:ODE_agent}
    \begin{aligned}
        &\dot{F}_t + A_t\transpose F_t + F_tA_t + Q_t - F_tK_t F_t = 0,\\
        &\dot{\mb{f}}_t + A_t\transpose \mb{f}_t + 2F_ta_t - 2Q_tr_t - F_tK_t \mb{f}_t = 0,
    \end{aligned}
    \end{equation}
    with terminal conditions \(F_T = 0\) and \(\mb{f}_T = 0\), where \(K_t := B_tS_t^{-1}B_t\transpose\in\mathbb{S}^{d\times d}\).
    
    Moreover, the ODE system~\eqref{eqn:ODE_agent} admits a unique solution on \([0,T]\) for any \(T>0\).
\end{prop}

Notably, the optimal control~\eqref{eqn:opt_ctrl} exhibits a certainty-equivalence structure: it uses the conditional mean \(\hat{X}_t\) as a surrogate for the unobserved state \(X_t\) and applies the optimal feedback law of the corresponding fully observed LQ control problem.

\subsection{Attacked dynamics and performance degradation}\label{sec:degrad}
We now incorporate an adversary into the partially observable control system, namely, an attacker that seeks to degrade the agent's performance by perturbing both the state and observation channels in dynamics~\eqref{eqn:X}--\eqref{eqn:Y}.

The corrupted state-observation processes \((X^c,Y^c)\), induced by the attack tuple \((\rho,\tau)\) and the agent's implemented control process \(u^a\), evolve according to
\begin{align}
    \label{eqn:Xc}
    \ud X^c_t &= (A_tX^c_t + B_tu^a_t+ a_t + \rho_t)\ud t + \sigma_V\ud V_t,\\
    \label{eqn:Yc}
    \ud Y^c_t &= (H_tX^c_t + h_t + \tau_t)\ud t + \sigma_W\ud W_t,
\end{align}
with \(X^c_0\sim N(x_0,R_0)\) and \(Y^c_0 = 0\). We assume that the attack processes satisfy the following measurability and integrability conditions: for any $t\in[0,T]$,
\begin{equation}
    \label{eqn:cond_rho_tau}
    \rho_t,\tau_t\in\F^{Y^c}_t,\ \ \E \int_0^T \|\rho_t\|^2\vee \|\tau_t\|^2\ud t<\infty.
\end{equation}
Here, \(\rho\) and \(\tau\) represent state- and observation-channel attacks, respectively:  \(\rho\) directly perturbs the state evolution while \(\tau\) biases sensor readings.
The adaptiveness with respect to the corrupted observation filtration $\{\F^{Y^c}_t\}$ reflects information constraint, that both the agent and the attacker observe only $Y^c$.

The agent is rational but non-strategic and vulnerable: it does not account for the possible presence of an attacker and acts under the belief that the system is secure \((\rho\equiv\tau\equiv 0)\). It therefore treats the observed path of \(Y^c\) as a genuine measurement outcome and applies the optimal feedback policy~\eqref{eqn:opt_ctrl} to its filtered estimate, leading to
\begin{equation}
    \label{eqn:ua}
    u^a_t := u^*(t,\hat{X}^a_t),
\end{equation}
where $\hat X^a$ is computed from $Y^c$ under the attack-free model. Namely, let $X^a$ be the state process in \eqref{eqn:Xc}--\eqref{eqn:Yc} with $\rho \equiv \tau \equiv 0$. Then $\hat{X}^a_t := \E(X^a_t|\F^{Y^c}_t)$ evolves 
\begin{multline}
    \label{eqn:Xa}
    \ud \hat{X}^a_t = (A_t \hat{X}^a_t + B_tu^a_t +  a_t)\ud t
    + R_tH_t\transpose \\(\sigma_W\sigma_W\transpose)^{-1}(\ud Y^c_t - (H_t\hat{X}^a_t + h_t)\ud t), \; \hat{X}^a_0 = x_0.
\end{multline}
In contrast, the attacker knows the agent's primary control objective and can anticipate the agent's optimal response \(u^*\), but does not directly observe the agent’s level of strategic awareness; in particular, it does not know whether the agent actively detect attacks. The attacker thereby aims to degrade the agent's closed-loop performance without altering the intrinsic randomness of the environment.  Let \(\hat{X}^c_t := \E(X^c_t|\F^{Y^c}_t)\) be the attack-aware filtered state. By the separation principle \cite{davis1977} and the Kalman-Bucy filter \cite{bain2009fundamentals}, it satisfies
\begin{multline}
    \label{eqn:X_hat_c}
    \ud \hat{X}^c_t = (A_t \hat{X}^c_t + B_tu^a_t + a_t + \rho_t)\ud t + R_tH_t\transpose \\(\sigma_W\sigma_W\transpose)^{-1}(\ud Y^c_t - (H_t\hat{X}^c_t + h_t + \tau_t)\ud t),
\end{multline}
with \(\hat{X}^c_0 = x_0\).
This is the filtered estimate that would be obtained by an observer who correctly accounts for the attack signals \(\rho\) and \(\tau\).

Importantly, the discrepancy \(\Delta X_t := \hat{X}^c_t - \hat{X}^a_t\) captures the informational asymmetry caused by the attack: \(\hat X^c\) is the attack-aware estimate, whereas \(\hat X^a\) is the estimate maintained by the vulnerable agent. 
This misalignment is exploited by the attacker to alter the agent's feedback behavior (cf.~\eqref{eqn:ua}).
Subtracting equation~\eqref{eqn:Xa} from equation~\eqref{eqn:X_hat_c} yields the dynamics of \(\Delta X\).
\begin{prop}
    \label{prop:DeltaX}
    The discrepancy process \(\Delta X_t\) satisfies
    \begin{multline}
        \label{eqn:Delta_X}
        \ud \Delta X_t = \Big([A_t - R_tH_t\transpose (\sigma_W\sigma_W\transpose)^{-1}H_t]\Delta X_t \\+ [\rho_t - R_tH_t\transpose (\sigma_W\sigma_W\transpose)^{-1}\tau_t]\Big)\ud t,\quad \Delta X_0 = 0.
    \end{multline}
    Under condition~\eqref{eqn:cond_rho_tau}, equation~\eqref{eqn:Delta_X} admits a unique \(\{\F^{Y^c}_t\}\)-adapted solution on \([0,T]\), almost surely.
    In particular, if \(\rho\) and \(\tau\) are deterministic, then \(\Delta X\) is deterministic as well.
\end{prop}

The whitened innovation process \(I^a\) associated with the agent's filter~\eqref{eqn:Xa} is defined by
\begin{equation}
    \label{eqn:Ia}
    \hspace{-0.5em}\ud I^a_t := (\sigma_W\sigma_W\transpose)^{-\frac{1}{2}}[\ud Y^c_t - (H_t\hat{X}^a_t + h_t)\ud t],\ I^a_0 = 0.
\end{equation}
It represents the normalized prediction error between the corrupted observation increment and the increment predicted by the agent's filter.
Plugging equation~\eqref{eqn:Yc} into equation~\eqref{eqn:Ia} yields the equivalent representation
\begin{multline}
    \label{eqn:I^a_W}
    \ud I^a_t = (\sigma_W\sigma_W\transpose)^{-\frac{1}{2}}[H_t(X^c_t - \hat{X}^a_t) + \tau_t]\ud t \\+ (\sigma_W\sigma_W\transpose)^{-\frac{1}{2}}\sigma_W\ud W_t.
\end{multline}
By construction, the associated filtrations satisfy \(\F^{\hat{X}^a}_t\subset \F^{I^a}_t = \F^{Y^c}_t\subset \F_t\).
The equality $\F^{I^a}_t = \F^{Y^c}_t$ follows because $I^a$ is constructed from $Y^c$, while $Y^c$ can be recovered from $I^a$ through \eqref{eqn:Ia} and the filtering equation \eqref{eqn:Xa}.
In particular, the innovation path \(I^a\) is available to the agent via \eqref{eqn:Ia}  and can be used to assess whether the observed measurements are statistically consistent with the attack-free model.

\begin{rem}[Universal filtering covariance]
    \label{rem:R}
    From Kalman-Bucy filtering theory \cite{bain2009fundamentals}, the conditional covariance matrices $\mathrm{cov}(X^a_t|\F^{Y^c}_t)$ and $\mathrm{cov}(X^c_t|\F^{Y^c}_t)$ both solve the same Riccati equation~\eqref{eqn:R}. This equation is independent of both the attack tuple \((\rho,\tau)\) and the control \(u\). Hence, once the model coefficients are fixed, $R_t$ can be solved without knowledge of $\rho, \tau, u$, and is treated as a fixed model coefficient in the following context.   
    With a slight abuse of notation, we use $R_t$ to denote the conditional covariance matrix for both $X^c$ and $X^a$.
\end{rem}

Anticipating the agent's use of \(u^a\) defined in \eqref{eqn:ua}, the attacker can evaluate the degradation of the agent’s primary task performance~\eqref{eqn:J_agent} through the induced cost.

\begin{defn}
    \label{defn:degrade}
    For a given attack tuple \((\rho,\tau)\), define
    \begin{equation}
        \label{eqn:def_D}
        \mc{D}(\rho,\tau) := \E \int_0^T [(X^c_t - r_t)\transpose Q_t (X^c_t - r_t) + (u^a_t)\transpose S_t u^a_t]\ud t.
    \end{equation}
     The actual performance degradation relative to the attack-free baseline is $\mc{D}(\rho, \tau)- \mc{D}(0,0)$. Since \(\mc{D}(0,0)\) is independent of $(\rho, \tau)$, subtracting it does not affect the optimal attack design developed in Sections~\ref{sec:det}--\ref{sec:adpt}. We therefore use $\mc{D}(\rho, \tau)$ as the performance degradation metric: a larger value of \(\mc{D}(\rho,\tau)\) corresponds to a more effective attack.
\end{defn}

\section{Detection framework and stealthiness characterization}\label{sec:detector} 

In this section, we consider the attacked dynamics~\eqref{eqn:Xc}--\eqref{eqn:Ia} and propose a likelihood-based detection mechanism.
The detector is constructed via a Girsanov-type argument and leads to a quantitative notion of attack stealthiness. We also interpret the detector and relate it to classical innovation-based detection methods.

\subsection{Likelihood-based detection and stealthiness}\label{subsec:detect}

Without prior knowledge about the agent’s strategic awareness , the attacker assumes that the agent is rational but vulnerable. At the same time, the attacker takes precautionary measures to ensure the stealthiness of its adversarial interventions, in case the agent employs a detection mechanism to monitor for evidence of attacks.

To quantify the stealthiness of attacks, a detection mechanism must be specified \textit{a priori}.
In this work, we adopt a likelihood-based detector constructed from the innovation process.
Given a sample path of \(I^a\), the log-likelihood of observing this trajectory under a candidate attack tuple \((\rho,\tau)\) provides a natural measure of statistical evidence for the presence of attacks. It forms the basis of our stealthiness criterion. The existence and explicit expression of this log-likelihood are given below.

\begin{prop}
    \label{prop:detection}
    Under condition~\eqref{eqn:cond_rho_tau}, the law of $I^a$ is absolutely continuous with respect to that of the $\R^m$-valued Brownian motion \(W^m:= (\sigma_W\sigma_W\transpose)^{-\frac{1}{2}}\sigma_W W\), i.e., \(\mu_{I^a}\ll\mu_{W^m}\). The log-likelihood $\ell(I^a;\rho,\tau) := \log\frac{\ud \mu_{I^a}}{\ud \mu_{W^m}}(I^a)$ is thus well-defined and admits the representation:
    \begin{align}
    \label{eqn:log_lkhd}
     &\ell(I^a;\rho,\tau)
     = \int_0^T (H_t\Delta X_t + \tau_t)\transpose (\sigma_W\sigma_W\transpose)^{-\frac{1}{2}}\ud I^a_t\\
     &- \frac{1}{2}\int_0^T (H_t\Delta X_t + \tau_t)\transpose (\sigma_W\sigma_W\transpose)^{-1}(H_t\Delta X_t + \tau_t)\ud t.
    \end{align}
\end{prop}
A larger value of \(\ell\) provides stronger statistical evidence that the observed innovation \(I^a\) is generated under the attacks \((\rho,\tau)\).
In the attack-free case,  i.e., \(\rho\equiv\tau \equiv 0\), one has \(\Delta X\equiv 0\) by equation~\eqref{eqn:Delta_X}, and hence \(\ell(I^a;\rho,\tau) = 0\) for all innovation trajectories.

The evaluation of \(\ell\) requires specifying a candidate attack tuple \((\rho,\tau)\).
While an attack-aware defender can compute \(\ell\) for different candidates, the true attack is unknown and must be inferred.
Although this paper focuses on the attacker’s design problem, where \((\rho,\tau)\) is treated as known, Remark~\ref{rem:interp_ell} elaborates how \(\ell\) can be used to construct a detector for an attack-aware defender based on the generalized likelihood ratio test (GLRT).

\begin{rem}
    \label{rem:interp_ell}
    For an attack-aware defender, a classical approach to attack detection is GLRT under the null hypothesis \(H_0: \rho\equiv\tau\equiv 0\) versus the alternative \(H_1: \rho\not\equiv 0 \text{ or }\tau\not\equiv 0\).
    Based on \(I^a\), the agent computes 
    estimators of attacks \(\hat{\rho}\) and \(\hat{\tau}\), leading to the log-generalized likelihood ratio (GLR) statistic \(\log\mathrm{GLR} := \ell(I^a;\hat{\rho},\hat{\tau}) - \ell(I^a;\rho\equiv 0,\tau\equiv 0)\).
    The null hypothesis \(H_0\) is rejected (i.e., the presence of attacks is confirmed) whenever \(\log \mathrm{GLR}\) exceeds a prescribed threshold.
\end{rem}

From the attacker’s perspective, treating \((\rho,\tau)\) as known, a smaller value of \(\ell(I^a;\rho,\tau)\) generally corresponds to reduced detectability of the attacks. This observation motivates the following notion of stealthiness, which captures average rather than pathwise detectability and yields a tractable stealthiness criterion.

\begin{defn}
    \label{defn:S}
    For a given attack tuple \((\rho,\tau)\), define its stealthiness as
    \begin{equation}
        \label{eqn:def_S}
        \mc{S}(\rho,\tau) := \E [\ell(I^a;\rho,\tau)].
    \end{equation}
\end{defn}
\begin{prop}
    \label{prop:stealthy}
    Under condition~\eqref{eqn:cond_rho_tau},
    \begin{multline}
        \label{eqn:stealthy_detector}
        \mc{S}(\rho,\tau) = \frac{1}{2}\E \int_0^T (H_t\Delta X_t + \tau_t)\transpose\\
        (\sigma_W\sigma_W\transpose)^{-1}(H_t\Delta X_t + \tau_t)\ud t.
    \end{multline}
\end{prop}

\subsection{An innovation-based detector}
The likelihood-based detector proposed in Section~\ref{subsec:detect} belongs to the broader class of innovation-based detectors \cite[Section~3.2.1]{zhang2021survey}. These detectors are motivated by a classical result in stochastic filtering: under the attack-free model, the innovation process \(I^a\) is an \(\{\F^{Y^c}_t\}\)-adapted Brownian motion  \cite{bain2009fundamentals}.

By comparing equation~\eqref{eqn:I^a_W} with the log-likelihood~\eqref{eqn:log_lkhd}, our proposed detector can be interpreted as identifying drift distortions in the observed sample path of \(I^a\), where attacks manifest themselves through the term \(H_t\Delta X_t+\tau_t\). In this sense, deviations from the Brownian structure of the innovation process serve as evidence of adversarial perturbations.

We next characterize the detectability of general innovation-based detectors in Proposition~\ref{prop:innov_BM}, highlighting the distinct roles played by the two attack components in Remark~\ref{rem:detactability}.
By situating our likelihood-based detector within this broader framework, we gain a clearer understanding of its capabilities and limitations.

\begin{prop}
    \label{prop:innov_BM}
    The innovation process \(I^a\) is an \(\{\F^{Y^c}_t\}\)-adapted Brownian motion if and only if 
    $
    \int_s^t \E(H_u\Delta X_u + \tau_u| \F^{Y^c}_s)\ud u = 0,\ \forall 0\leq s\leq t\leq T.
    $
    In particular, if \(\rho\) and \(\tau\) are deterministic, and \(A_{t_1}\) commutes with \(A_{t_2}\) for any \(t_1,t_2\in[0,T]\), this condition is equivalent to 
    \begin{equation}
        \label{eqn:innov_cond}
        H_te^{\int_0^t A_u\ud u}\int_0^t e^{-\int_0^s A_u\ud u} \rho_s\ud s + \tau_t = 0,\ \forall t\in[0,T].
    \end{equation}
\end{prop}

\begin{rem}[Detectability of \(\rho\) \textit{vs.} \(\tau\)]
    \label{rem:detactability}
    Let \(A_t\equiv A\) be constant.
    For deterministic attacks \(\rho\) and \(\tau\) satisfying~\eqref{eqn:innov_cond}, innovation-based detectors fail to detect the corresponding perturbations.
    The two attack components exhibit distinct detectability properties. 
    When \(\rho \equiv 0\), condition~\eqref{eqn:innov_cond} holds if and only if \(\tau\equiv 0\), whereas the converse does not hold.
    In particular, pure state attacks $\rho$ may remain undetectable, while pure observation attacks $\tau$ are always detectable. This asymmetry indicates that state attacks are intrinsically more stealthy than observation attacks under innovation-based detection.
\end{rem}

\begin{rem}[Comparison with the \(\chi^2\)-detector]
\label{rem:compare}
A classical innovation-based detector is the \(\chi^2\)-detector \cite[Section~7]{willsky1976survey}.
Its construction is motivated by the fact that, under the attack-free model, \(I^a\) is an \(\{\F^{Y^c}_t\}\)-adapted Brownian motion, and over a time window \([t,t+wh]\) with step size \(h\), the statistic 
$$
\chi^2_{[t,t+wh]} := \frac{1}{h}\sum_{j=0}^{w-1} \|I^a_{t+(j+1)h} - I^a_{t+jh}\|^2
$$
follows a standard \(\chi^2\)-distribution with  \(wm\) degrees of freedom.

In the continuous-time limit as \(h\to 0\), the \(\chi^2\)-detector effectively tests the quadratic variation of \(I^a\), which is invariant under drift perturbations \((\rho,\tau)\); see \eqref{eqn:I^a_W}. In contrast, the log-likelihood~\eqref{eqn:log_lkhd} depends explicitly on the drift distortion induced by the attacks.
Although both detectors are innovation-based, our likelihood-based detector exploits the full path-space measure, whereas the \(\chi^2\)-detector relies only on a projection of the path law onto a low-dimensional summary statistic.
\end{rem}

\section{Optimal design of deterministic stealthy attacks}\label{sec:det}

The preceding section provides two quantitative criteria for attack design: the induced performance degradation $\mc{D}$ and the likelihood-based stealthiness measure $\mc{S}$. We now combine these criteria into an optimization problem for the attacker.

We first consider deterministic attacks, which are fixed before the system evolves and therefore represent an offline design problem. This setting captures the trade-off between effectiveness and detectability while leading to a tractable LQ formulation.
To this end, the admissible attacks \((\rho,\tau)\) are required to take values in
\begin{equation}
    \mathscr{A}^{\mathrm{det}} := \{(\rho,\tau):\rho\in C([0,T];\R^d),\ \tau\in C([0,T];\R^m)\}.
\end{equation}
Thus $\rho, \tau$ do not adapt to the realized system trajectories. Such situations may arise when the attacker has costly real-time access to the system.

Accordingly, the attacker considers the deterministic optimization problem:
\begin{equation}
    \label{eqn:det_obj}
    \inf_{(\rho,\tau)\in\mathscr{A}^{\mathrm{det}}} \mc{S}(\rho,\tau)- \lambda\mc{D}(\rho,\tau) + \frac{1}{2}\int_0^T \rho_t\transpose P_t\rho_t\ud t,
\end{equation}
where \(\lambda \geq 0\) and \(P\in C([0,T];\mathbb{S}_{++}^{d\times d})\).
The intensity \(\lambda\) controls the relative weight assigned to performance degradation, while \(P\) penalizes the magnitude of the state attack \(\rho\). 

This penalization term is motivated by earlier observation (cf. Remark~\ref{rem:detactability}) that state attacks are intrinsically more stealthy than observation attacks. It reflects the practical distinction: sensor spoofing is typically inexpensive but more easily detectable, whereas direct manipulation of the system dynamics is more costly but harder to detect.
In addition, this term plays a technical role in ensuring well-posedness of the optimization problem once its LQ structure is revealed (cf. Remark~\ref{rem:attk_cost}). Overall, the objective in~\eqref{eqn:det_obj} balances the effectiveness, detectability, and implementation cost of the attacks.

The following result provides a simplified expression for \(\mathcal{D}\) in~\eqref{eqn:det_obj}, revealing its underlying LQ structure.
\begin{prop}
    \label{prop:degrad_det}
    Let $\Psi_t := \mathrm{Concat}(X_t^c,\hat{X}^a_t)$, and define
    \begin{align} 
        &m_t := \E \Psi_t=: \mathrm{Concat}(m^c_t,m^a_t)\in\R^{2d},\\
        &\Sigma_t := \mathrm{cov}(\Psi_t) =: \begin{bmatrix}
            \Sigma^{cc}_t & \ast\\
            \ast & \Sigma^{aa}_t
        \end{bmatrix}\in\mathbb{S}^{2d\times 2d}_{+}.
    \end{align}
    If \((\rho,\tau)\in \mathscr{A}^{\mathrm{det}}\), then \(\Psi_t\) is Gaussian \(\forall t\in[0,T]\), and 
    \begin{multline}
        \label{eqn:D_integral}
        \mc{D}(\rho,\tau) = \int_0^T \mathrm{Tr}(Q_t  \Sigma^{cc}_t) +(m^c_t - r_t)\transpose Q_t (m^c_t - r_t) \\
        +  
        \mathrm{Tr}(F_tK_t F_t\Sigma^{aa}_t) + (F_tm^a_t + \tfrac12 \mb{f}_t)\transpose K_t  (F_tm^a_t + \tfrac12 \mb{f}_t)\ud t.
    \end{multline}
    The mean \(m\) and covariance \(\Sigma\) satisfy
    \begin{equation}
        \label{eqn:m_Sigma}
        \begin{aligned}
        \dot{m}_t &= \Big(\mathcal{A}_t m_t + \mathcal{B}_t\, \mathrm{Concat}(\rho_t, \tau_t) + \alpha_t\Big),\\ \dot{\Sigma}_t &= \Sigma_t \mathcal{A}_t\transpose + \mathcal{A}_t\Sigma_t + \mathcal{V}_t\mathcal{V}_t\transpose.
        \end{aligned}
    \end{equation}
    Here, the coefficients and initial conditions are given by
    \begin{equation}
        \label{eqn:aux_quant}
        \begin{aligned}
        &\mathcal{A}_t := \begin{bmatrix}
           A_t & -K_t F_t\\
           \Theta_t &A_t - K_t F_t - \Theta_t
        \end{bmatrix},\ \mathcal{V}_t := \begin{bmatrix}
           \sigma_V & 0\\
           0 & \mc{T}_t\sigma_W
        \end{bmatrix},\\
        &\mathcal{B}_t := \begin{bmatrix}
           I & 0\\
           0 & \mc{T}_t
        \end{bmatrix},\ \Sigma_0 = \begin{bmatrix}
            R_0 & 0\\
            0 & 0
        \end{bmatrix},\ m_0 = \mathrm{Concat}(x_0,x_0),\\
        &\alpha_t := \mathrm{Concat}(a_t - \tfrac12 K_t \mb{f}_t, a_t - \tfrac12 K_t \mb{f}_t),\\
        &\mc{T}_t := R_tH_t\transpose (\sigma_W\sigma_W\transpose)^{-1},\ \Theta_t := R_tH_t\transpose (\sigma_W\sigma_W\transpose)^{-1}H_t.
        \end{aligned}
    \end{equation}
\end{prop}

By Proposition~\ref{prop:degrad_det}, under deterministic attacks, \(\Sigma^{cc}\) and \(\Sigma^{aa}\) are both independent of $(\rho, \tau)$. Hence, \(\mc{D}\) reduces, up to attack-independent terms, to a time integral of a quadratic functional of \(m^c\) and \(m^a\). Moreover, since \(\Delta X\) is deterministic (cf. Proposition~\ref{prop:DeltaX}), tower property yields
\begin{equation}
    \label{eqn:DeltaX_det}
    \Delta X_t = \E (\Delta X_t) = \E\hat{X}^c_t - \E\hat{X}^a_t = m^c_t - m^a_t.
\end{equation}
Together with \eqref{eqn:stealthy_detector}, this shows that the stealthiness term \(\mc{S}\) also admits a quadratic structure.

Therefore, the deterministic attack design problem~\eqref{eqn:det_obj} reduces to a deterministic LQ optimal control problem, with state \(m\) evolving according to the linear dynamics~\eqref{eqn:m_Sigma} and \((\rho,\tau)\) as open-loop controls.  Theorem~\ref{thm:solution_det} provides a semi-explicit solution for the optimal attack tuple \((\rho^*,\tau^*)\) via Pontryagin's maximum principle \cite{mangasarian1966sufficient,seierstad1977sufficient},

\begin{thm}
    \label{thm:solution_det}
    For the optimization problem~\eqref{eqn:det_obj}, the optimal deterministic attacks \((\rho^*,\tau^*)\in\mathscr{A}^{\mathrm{det}}\) are given by
    \begin{align}
    \label{eqn:opt_det_rho}
    \rho^*_t &= -P_t^{-1}(F^c_tm^c_t + F^a_tm^a_t + \mb{f}^\rho_t),\\
    \label{eqn:opt_det_tau}
    \tau^*_t &= -(H_tR_tG^c_t + H_t)m^c_t - (H_tR_tG^a_t - H_t)m^a_t \\
    &\qquad- H_tR_t\mb{g}^\tau_t,
    \end{align}
    where the matrix-valued functions \(F^c,F^a,G^c,G^a\in C([0,T];\mathbb{R}^{d\times d})\) and the vector-valued functions \(\mb{f}^\rho,\mb{g}^\tau\in C([0,T];\mathbb{R}^{d})\) satisfy
    \begin{equation}
    \label{eqn:ODE_det}
    \begin{aligned}
        &\dot{F}^c_t + F^c_tA_t + A_t\transpose F^c_t - F^c_tP_t^{-1}F^c_t - F^a_t\Lambda_tG^c_t - 2\lambda Q_t = 0,\\
        &\dot{F}^a_t + F^a_t(A_t-K_tF_t) + A_t\transpose F^a_t - F^c_tK_t F_t - F^c_tP_t^{-1}F^a_t  \\
        &\quad - F^a_t\Lambda_tG^a_t= 0,\\
        &\dot{G}^c_t + G^c_tA_t  + (A_t\transpose - F_tK_t)G^c_t - F_tK_tF^c_t - G^a_t \Lambda_t G^c_t\\
        &\quad - G^c_tP_t^{-1}F^c_t = 0,\\
        &\dot{G}^a_t + G^a_t(A_t-K_tF_t) + (A_t\transpose-F_tK_t) G^a_t - F_tK_tF^a_t \\
        &\quad - G^c_tK_t F_t  - G^c_t P_t^{-1}F^a_t  - G^a_t \Lambda_t G^a_t - 2\lambda F_tK_t F_t= 0,\\
        &\dot{\mb{f}}^\rho_t + (A_t\transpose- F^c_tP_t^{-1})\mb{f}^\rho_t + F^c_t(a_t - \tfrac12 K_t \mb{f}_t) - F^a_t\Lambda_t\mb{g}^\tau_t\\
        &\quad + F^a_t(a_t - \tfrac12 K_t \mb{f}_t)   + 2\lambda Q_tr_t = 0,\\
        &\dot{\mb{g}}^\tau_t+ (A_t\transpose-F_tK_t)\mb{g}^\tau_t- G^c_tP_t^{-1}\mb{f}^\rho_t - G^a_t\Lambda_t\mb{g}^\tau_t - F_tK_t \mb{f}^\rho_t \\
        &\quad + G^c_t(a_t - \tfrac12 K_t \mb{f}_t) + G^a_t(a_t - \tfrac12 K_t \mb{f}_t) - \lambda F_tK_t \mb{f}_t = 0,
    \end{aligned}
\end{equation}
    with terminal conditions \(F^c_T = F^a_T = G^c_T = G^a_T = 0\) and \(\mb{f}^\rho_T =\mb{g}^\tau_T = 0\), where \(\Lambda_t := R_tH_t\transpose (\sigma_W\sigma_W\transpose)^{-1}H_tR_t\).

    Moreover, the ODE system~\eqref{eqn:ODE_det} admits a unique solution on \([0,T]\), provided that
    \begin{equation}
        \label{eqn:exist_interval}
        T < \frac{\pi/2}{\sqrt{[2\lambda(\|Q\|\vee \|FKF\|)+ b_{\mc{P}}](\|P^{-1}\|\vee \|\Lambda\| + b_{\mc{P}})}},
    \end{equation}
    where \( b_{\mc{P}} := \|A\|\vee \|A-KF\| + \|KF\|\).
    In particular, when \(\lambda = 0\), the ODE system~\eqref{eqn:ODE_det} admits a unique solution on \([0,T]\) for any \(T>0\).
\end{thm}

Theorem~\ref{thm:solution_det} establishes the local solvability of the Riccati-type ODE system~\eqref{eqn:ODE_det}, following an argument similar to that in~\cite{hu2025strategic}. The existence interval depends on the trade-off parameter $\lambda$: placing greater weight on performance degradation may shorten the existence interval of the solution to the ODE system.

\begin{rem}
    \label{rem:attk_cost}
    From a technical perspective, the problem~\eqref{eqn:det_obj} becomes ill-posed without the state-attack cost (\(P \equiv 0\)). 
    In that case, the objective~\eqref{eqn:det_obj} depends on  
    \(\rho\) solely through the mean states \(m^c\) and \(m^a\) (cf. equations~\eqref{eqn:m_Sigma}--\eqref{eqn:DeltaX_det}), and thus 
    lacks direct regularization on \(\rho\). Consequently, the Hamiltonian does not admit a minimizer. Introducing the quadratic term \(\rho_t\transpose P_t \rho_t\) restores well-posedness by regularizing the state attack \(\rho\), in addition to its modeling motivation.
\end{rem}

\section{Optimal design of adaptive stealthy attacks}\label{sec:adpt}
Complementary to the deterministic setting in Section~\ref{sec:det}, we study the optimal design of stealthy attacks when the attacker can adapt to the realized system evolution. This leads to a stochastic control problem under partial observation with an endogenous information structure.

In contrast to existing work, where attacks are typically open-loop or constrained by exogenous detection statistics, the attacker here operates in feedback form while explicitly accounting for detectability through the innovation. This combination introduces significant technical challenges, as the attacks affect both the system dynamics and the information available for future decisions. 
We develop a \emph{hierarchical optimization framework}, combined with the separation principle, to obtain semi-explicit characterizations of optimal adaptive attacks.

In this setting, the attack tuple \((\rho,\tau)\) is restricted to the following admissible class:
\begin{multline}
    \label{eqn:adm_adpt}
    \mathscr{A}^{\mathrm{adap}} := \Big\{(\rho,\tau):\rho_t \in \F^{Y^c}_t,\ \forall t\in[0,T],\\
    \E \int_0^T \|\rho_t\|^2\ud t<\infty,\ \tau\in C([0,T];\R^m)\Big\},
\end{multline}
which constitutes a further restriction of condition~\eqref{eqn:cond_rho_tau}.
The state attack \(\rho\) is constructed adaptively based on the progressively revealed observation \(Y^c\), whereas the observation attack \(\tau\) is restricted to be deterministic. 

This asymmetric structure is both natural and essential.
The state attack \(\rho\) acts as an external input to the system dynamics and naturally admits feedback adaptation based on observed outputs. 
In contrast, the observation attack \(\tau\) directly alters the measurement process from which both the agent and the attacker extract information. 
Allowing \(\tau\) to depend on the corrupted observations $Y^c$ would thus create a self-referential information structure, where the attacker designs the observation distortion $\tau$ based on data that is itself affected by this distortion.
For this reason, \(\tau\) is more naturally modeled as a pre-designed signal, corresponding to a fixed spoofing or biasing mechanism implemented at the sensor or communication layer. From a technical perspective, such dependence would result in a control-dependent innovation process and invalidate the separation principle. Restricting \(\tau\) to be deterministic avoids this issue while still capturing realistic spoofing or biasing mechanisms; see Remark~\ref{rem:tau_det} for detailed discussion.

Under this admissible class $\mathscr{A}^{\mathrm{adap}}$, the attacker solves 
\begin{equation}
    \label{eqn:adap_obj}
    \inf_{(\rho,\tau)\in\mathscr{A}^{\mathrm{adap}}} \mc{S}(\rho,\tau)- \lambda\mc{D}(\rho,\tau) + \frac{1}{2}\E\int_0^T \rho_t\transpose P_t\rho_t\ud t,
\end{equation}
where \(\mc{D}\) and \(\mc{S}\) are given by equations~\eqref{eqn:def_D} and~\eqref{eqn:stealthy_detector}.
This formulation~\eqref{eqn:adap_obj} is consistent with the deterministic case~\eqref{eqn:det_obj} and retains the same interpretation.

\subsection{Solution strategy: hierarchical optimization}
A direct solution of the adaptive problem is challenging due to the coupling between control, state dynamics, and information structure. To address this difficulty, we develop a hierarchical optimization framework.

\textbf{Step (a):} we first fix $\tau$ and characterize the optimal adaptive state attack \(\rho^*(\tau)\). In this step, note that the discrepancy process \(\Delta X\) is \(\{\F^{Y^c}_t\}\)-adapted whenever \((\rho,\tau)\in \mathscr{A}^{\mathrm{adap}}\) (cf. Proposition~\ref{prop:DeltaX}). This allows us to interpret the optimization over $\rho$, for fixed $\tau$, as a partially observed control problem. In this formulation, \(X^c\) plays the role of a latent state, while $Y^c$, the filtered state $\hat{X}^a$, and the discrepancy \(\Delta X\) serve as observable quantities. When the separation principle \cite{davis1977} applies, the optimal $\rho^\ast(\tau)$ admits a feedback representation in terms of \(\hat{X}^c,Y^c,\hat{X}^a,\Delta X\), and problem~\eqref{eqn:adap_obj} reduces further to a fully observed Markovian control problem. The following theorem provides the corresponding characterization.

\begin{thm}
    \label{thm:opt_adap_rho}
    Fix $\tau \in C([0,T];\R^m)$. The optimal state attack $\rho^*(\tau)$ at time \(t\), given the filtered states $\hat{X}^c_t = x^c$, $\hat{X}^a_t = x^a$, and $\Delta X_t = \Delta x$, admits a linear structure
    \begin{equation}
        \label{eqn:opt_adp_rho}
    \rho^*(t,x^c,x^a,\Delta x;\tau) = -P_t^{-1}(e_{x^c} + e_{\Delta x})\transpose(F^\phi_t\phi + \mb{f}^\phi_t),
    \end{equation}
    where \(\phi := \mathrm{Concat}(x^c,x^a,\Delta x) \in \R^{3d}\), and \(F^\phi\in C([0,T];\mathbb{S}^{3d\times 3d})\), \(\mb{f}^\phi\in C([0,T];\mathbb{R}^{3d})\) satisfy a coupled system of Riccati-type differential equations
    \begin{equation}
    \label{eqn:ODE_adap_rho}
    \begin{aligned}
        &\dot{F}^\phi_t - F^\phi_tO_tF^\phi_t + F^\phi_tD^\phi_t + (D^\phi_t)\transpose F^\phi_t + 2Q^\phi_t = 0,\\
        &\dot{\mb{f}}^\phi_t - F^\phi_tO_t\mb{f}^\phi_t + F^\phi_t(d^\phi_t + d^\tau_t) + (D^\phi_t)\transpose\mb{f}^\phi_t + \ell^\phi_t + \ell^\tau_t = 0,
    \end{aligned}
    \end{equation}
    with terminal conditions \(F^\phi_T = 0\) and \(\mb{f}^\phi_T = 0\).
    Here, the \(\tau\)-independent coefficients are given by
    \begin{equation} 
    \begin{aligned}
    &D^\phi_t := \begin{bmatrix}
        A_t & -K_tF_t & 0\\
        \Theta_t & A_t - K_tF_t - \Theta_t & 0\\
        0 & 0 & A_t - \Theta_t
    \end{bmatrix},\\
    &d^\phi_t := \mathrm{Concat}(a_t - \tfrac12 K_t \mb{f}_t,a_t - \tfrac12 K_t \mb{f}_t,0),\\
    &e_{x^c}\transpose := [ I_d \; 0 \; 0], \ e_{x^a}\transpose := [ 0 \; I_d \; 0] ,\ e_{\Delta x}\transpose := [ 0 \; 0 \; I_d]\in\R^{d\times 3d},\\
    & O_t := (e_{x^c} + e_{\Delta x})P_t^{-1}(e_{x^c} + e_{\Delta x})\transpose\in\mathbb{S}^{3d\times 3d}, \\
    &\ell^\phi_t := 2\lambda e_{x^c}Q_tr_t -\lambda e_{x^a}F_tK_t\mb{f}_t\in \R^{3d},\\
    & Q^\phi_t := \tfrac12e_{\Delta x}H_t\transpose (\sigma_W\sigma_W\transpose)^{-1}H_te_{\Delta x}\transpose -\lambda e_{x^a}F_tK_tF_te_{x^a}\transpose \\
    &\qquad\qquad - \lambda e_{x^c}Q_te_{x^c}\transpose\in\mathbb{S}^{3d\times 3d}.
    \end{aligned}
    \end{equation}
    The \(\tau\)-dependent coefficients are defined as
    \begin{equation}
    \label{eqn:def_d_ell}
    \begin{aligned}
    &d^\tau_t := \mathrm{Concat}(0,\mc{T}_t\tau_t,-\mc{T}_t\tau_t)\in\R^{3d},\\
    &\ell^\tau_t := e_{\Delta x}H_t\transpose(\sigma_W\sigma_W\transpose)^{-1}\tau_t\in \R^{3d}.
    \end{aligned}
    \end{equation} 
    Moreover, the ODE system~\eqref{eqn:ODE_adap_rho} admits a unique solution on \([0,T]\) for small enough \(T\).
\end{thm}

This result shows that, despite partial observation and endogenous information, the optimal adaptive attack retains a tractable structure. The feedback form reflects how the attacker dynamically exploits the mismatch between the true and perceived system states.

As in the deterministic case, the existence interval depends on the parameter \(\lambda\). A larger value of $\lambda$, corresponding to a stronger emphasis on performance degradation, may shorten the existence interval of the solution to system~\eqref{eqn:ODE_adap_rho}.

\begin{rem}
    \label{rem:tau_det}
    A key technical difficulty in establishing Theorem~\ref{thm:opt_adap_rho} arises from the \(\tau\)-dependence of the innovation Brownian motion associated with equations~\eqref{eqn:Xc}--\eqref{eqn:Yc}.
    If \(\rho\) and \(\tau\) are optimized simultaneously, the driving innovation Brownian motion would become controlled, violating the conditions required for the separation principle.
    Restricting \(\tau\) to be deterministic avoids this issue, and enables the hierarchical optimization framework, where \(\rho\) is only optimized under fixed \(\tau\). 
    
    In general, allowing \(\tau\) to be fully adaptive would lead to a substantially more complex control problem, potentially involving measure-valued controls \cite{fuhrman2026optimal}, which is beyond the scope of this paper and is left for future studies.
\end{rem}

\begin{rem}
    \label{rem:rho_star}
    The optimality of \(\rho^*(\tau)\) holds, for any fixed \(\tau\), within a restricted class of admissible controls that depend on the histories of \(Y^c,\hat{X}^c,\hat{X}^a\) and \(\Delta X\). This class can be explicitly constructed analogously to equation~\eqref{eqn:adm_separation} and \cite[Definition~5.3.1]{davis1977}, and is omitted for brevity.  
\end{rem}

\textbf{Step (b):} we determine the optimal observation attack \(\tau^*\) induced by \(\rho^*(\tau)\).
Substituting $\rho^\ast(\tau)$ into the objective~\eqref{eqn:adap_obj} reduces the problem to a deterministic LQ control problem in $\tau$.

\begin{thm}
    \label{thm:opt_adap_tau}
    The optimal observation attack \(\tau^*\) is
    \begin{equation}
        \label{eqn:opt_adap_tau}
        \tau^*_t =  - (Q^F_{t})\transpose (F^\tau_{t} \mb{f}^\phi_{t} + \mb{f}^\tau_{t} + \Phi_0) - (\sigma_W\sigma_W\transpose)G_{t}\transpose \mb{f}^\phi_{t},
    \end{equation}
    where \(F^\tau\in C([0,T]; \mathbb{S}^{3d \times 3d})\) and \(\mb{f}^\tau\in C([0,T]; \R^{3d})\) solve the system
    \begin{equation}
    \label{eqn:ODE_U}
    \begin{aligned}
        &\dot{F}^\tau_t+ F^\tau_t Q^F_{t}(\sigma_W\sigma_W\transpose)^{-1}(Q^F_{t})\transpose F^\tau_t + (O_{t} + G_{t}\sigma_W\sigma_W\transpose G_{t}\transpose) \\
        & \quad - F^\tau_t [(D^\phi_{t})\transpose - F^\phi_{t}O_{t} - Q^F_{t}G_{t}\transpose] \\
        &\quad - [(D^\phi_{t})\transpose - F^\phi_{t}O_{t} - Q^F_{t}G_{t}\transpose]\transpose F^\tau_t = 0,\\
        &\dot{\mb{f}}^\tau_t - F^\tau_t[F^\phi_{t}d^\phi_{t} + \ell^\phi_{t} - Q^F_{t}(\sigma_W\sigma_W\transpose)^{-1}(Q^F_{t})\transpose(\mb{f}^\tau_t + \Phi_0)] \\
        &\quad\quad + [G_{t}(Q^F_{t})\transpose - D^\phi_{t} + O_{t}\transpose (F^\phi_{t})\transpose](\mb{f}^\tau_t + \Phi_0) - d^\phi_{t} = 0,
    \end{aligned}
    \end{equation}
    with initial conditions \(F^\tau_0 = 0\) and \(\mb{f}^\tau_0 = 0\).
    The coefficients are defined as
    \begin{equation}
        \begin{aligned}
        &\Phi_0 := \mathrm{Concat}(\hat{X}^c_0,\hat{X}^a_0,\Delta X_0) \in\R^{3d},\\
        &G_t := \mathrm{Concat}(0_{d\times m},\mc{T}_t,-\mc{T}_t)\in \R^{3d\times m},\\
        &Q^F_{t} := F^\phi_{t}G_{t}\sigma_W\sigma_W\transpose + e_{\Delta x}H_{t}\transpose\in \R^{3d\times m}.
        \end{aligned}
    \end{equation}
    Moreover, the ODE system~\eqref{eqn:ODE_U} admits a unique solution on \([0,T]\) for small enough \(T\).
\end{thm}

Combining Theorems~\ref{thm:opt_adap_rho} and~\ref{thm:opt_adap_tau} yields a constructive procedure for computing the optimal adaptive attack:
\vspace{-1em}
\begin{enumerate}[label=(\roman*)]
    \item solve~\eqref{eqn:ODE_adap_rho} for \(F^\phi\);
    \item solve~\eqref{eqn:ODE_U} for \((F^\tau, \mb{f}^\tau)\) and obtain $\tau^\ast$ from~\eqref{eqn:opt_adap_tau};
    \item solve~\eqref{eqn:ODE_adap_rho} for \(\mb{f}^\phi\) under \(\tau^*\);
    \item derive \(\rho^*(\tau^*)\) from eq.~\eqref{eqn:opt_adp_rho}.
\end{enumerate}

Overall, the hierarchical construction preserves tractability without eliminating feedback adaptation: the optimal state attack exploits the progressively revealed information, while the observation attack is optimized at the outer level to balance stealthiness and effectiveness.

\section{Numerical experiments}\label{sec:numerics}
In this section, we compute the optimal deterministic and adaptive attacks \((\rho^*,\tau^*)\) by solving the ODE systems~\eqref{eqn:R}, \eqref{eqn:ODE_agent}, \eqref{eqn:ODE_det}, \eqref{eqn:ODE_adap_rho} and~\eqref{eqn:ODE_U}.
All ODEs are solved using an eighth-order Runge–Kutta method on a fine temporal grid. The underlying stochastic dynamics are simulated using the Euler scheme over the time horizon $[0,T]$, discretized into $N_T = 1,000$ uniform subintervals. Performance degradation and stealthiness are estimated via Monte Carlo with 25,000 independent trajectories.

\subsection{One-dimensional mean-reverting dynamics}\label{subsec:1d}

We first consider a one-dimensional model with \(d = c = m = p = q = 1\) over the time horizon \([0,0.5]\), with the following model parameters
\begin{equation}
    \begin{aligned}
    &A \equiv -1,\ B \equiv 1,\ H \equiv 1,\ a \equiv h \equiv 0,\ x_0 = 0.5,\ R_0 = 0,\\
    &\sigma_V = 0.6,\ \sigma_W = 0.4,\ P\equiv S \equiv 1,\ Q \equiv 10,\ r\equiv 0.
    \end{aligned}
\end{equation}
The system parameters correspond to a mean-reverting state process toward zero, and the agent aims to keep the state close to zero.

Figure~\ref{fig:1d_det} (resp. Figure~\ref{fig:1d_adpt}) shows the trajectories of \(X^c\) and \(\hat{X}^a\) under optimal deterministic (resp. adaptive) attacks \((\rho^*,\tau^*)\) for different values of the trade-off parameter \(\lambda\).
When \(\lambda = 0\), both optimal attacks vanish \(\rho^*\equiv\tau^*\equiv 0\), reflecting the absence of adversarial incentives.
Larger values of \(\lambda\) lead to attacks of greater magnitudes, inducing a visible bias in the state trajectory \(X^c\) but not in \(\hat{X}^a\). This highlights the stealthy nature of the attack: the system is significantly perturbed without being detected through the agent’s filtering mechanism.

The top-right panel illustrates the trade-off between effectiveness and stealthiness: as \(\lambda\) increases, the optimal attacks are more effective but easier to detect.
Notably, stealthiness increases by orders of magnitude. 
Although its Monte Carlo estimate is extremely small, the associated variance is of an even lower order, indicating high statistical accuracy.

\begin{figure}[htbp]
    \centering
    \includegraphics[width=0.9\linewidth]{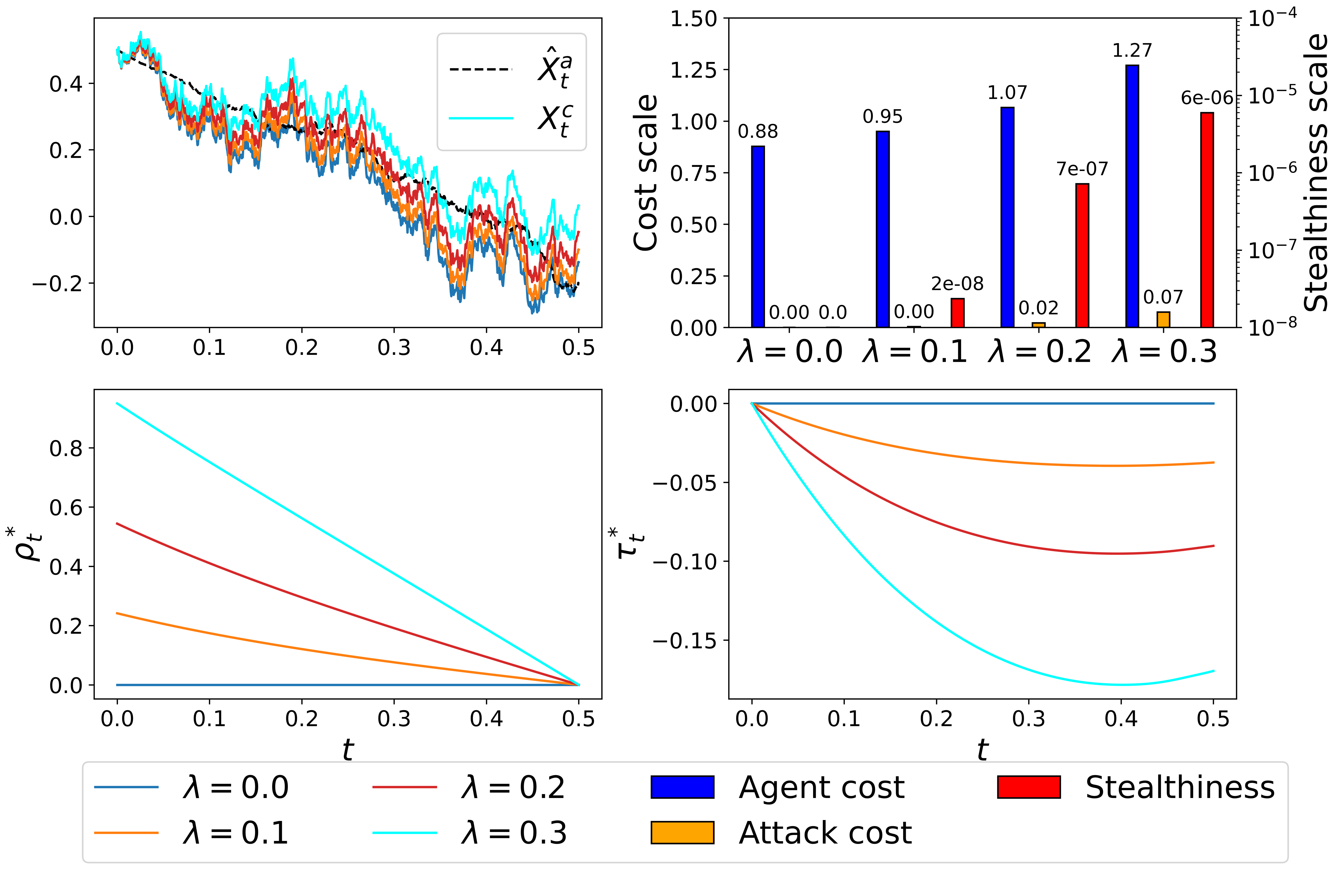}
    \caption{Trajectories under optimal deterministic attacks \((\rho^*,\tau^*)\) for different values of \(\lambda\) in the 1D mean-reverting model. The agent's filter \(\hat{X}^a\) is plotted only for \(\lambda = 0\) due to minimal variation across cases.}
    \label{fig:1d_det}
\end{figure}

\begin{figure}[htbp]
    \centering
    \includegraphics[width=0.9\linewidth]{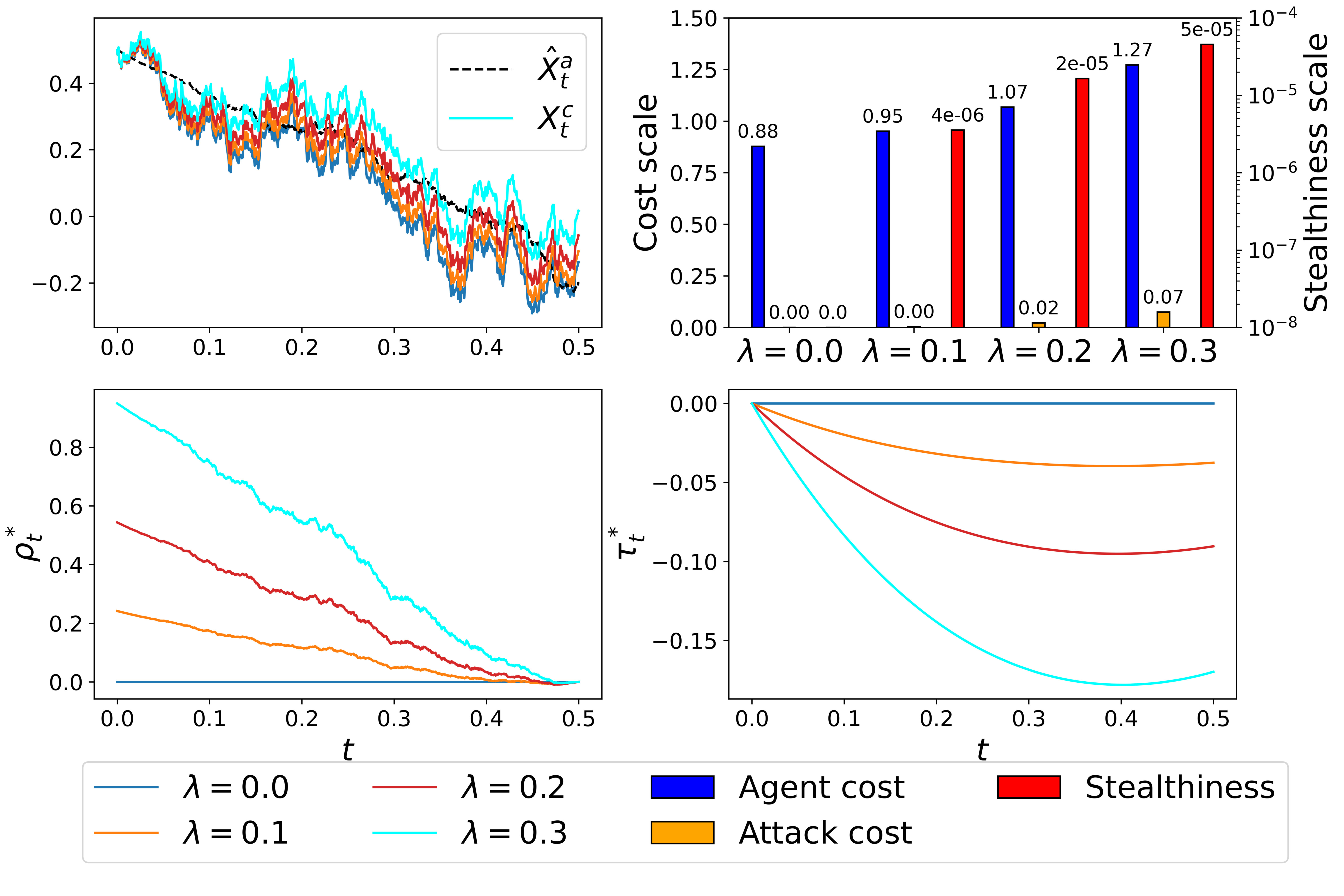}
    \caption{Trajectories under optimal adaptive attacks \((\rho^*,\tau^*)\) for different values of \(\lambda\) in the 1D mean-reverting model. 
    The agent's filter \(\hat{X}^a\) is plotted only for \(\lambda = 0\) due to minimal variation across cases.
    }
    \label{fig:1d_adpt}
\end{figure}

Figure~\ref{fig:comp} compares the proposed optimal deterministic and adaptive attacks with several alternative attack strategies, including a random Gaussian attack where \(\rho\) and \(\tau\) are independent standard Gaussian white noises, and a sinusoidal attack where \(\rho_t = -\tau_t = \sin (8\pi t)\).
Our attack designs produce visibly large differences in the trajectories of \(X^c\) and \(I^a\) compared to the baseline case (\(\lambda = 0\)), while more effectively increasing the agent's expected cost yet remaining harder to detect.
In Figure~\ref{fig:comp}, we also plot the detectability residual for deterministic attacks, defined as the residual of equation~\eqref{eqn:innov_cond}, which measures detectability under innovation-based detectors.
As observed, the optimal deterministic attack remains stealthy by keeping the detectability residual consistently low.

Finally, we compare the optimal deterministic and adaptive attacks: the difference in \(\tau^*\) is small, while the state attacks \(\rho^*\) differ more significantly. The adaptive one exhibits higher variability, reflecting its dependence on the observed trajectory. In this example, the optimal adaptive attack achieves greater performance degradation due to its ability to exploit real-time information.

\begin{figure}[htbp]
    \centering
    \includegraphics[width=0.9\linewidth]{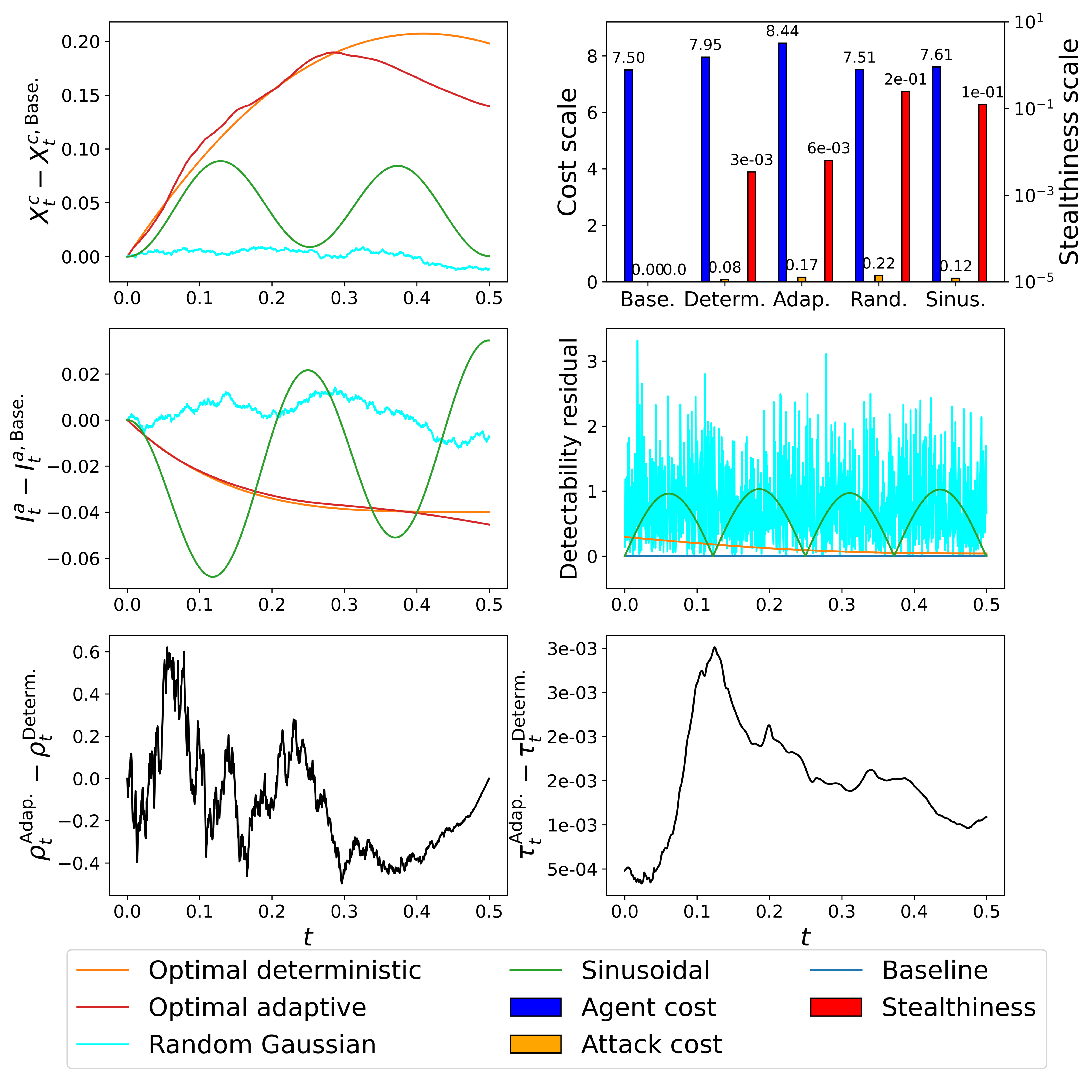}
    \caption{Comparison of different attack strategies in the 1D mean-reverting model.
    The proposed deterministic and adaptive attacks (with \(\lambda = 0.3\)) are compared with Gaussian and sinusoidal perturbations.
    Modified parameters \(\sigma_V = \sigma_W =1,R_0 = 2\) are used for improved visualization.}
    \label{fig:comp}
\end{figure}

\subsection{Two-dimensional tracking objective}\label{subsec:2d}
We next consider a two-dimensional model with \(d = c = m = p = q = 2\), on the time horizon \([0,0.5]\) with the following model parameters
\begin{equation}
\begin{aligned}
    &A \equiv 0,\ B \equiv P\equiv I_2,\ a \equiv h \equiv 0, \ \sigma_W = 0.1I_2,\ S = 0.5I_2,\\
    &R_0 = 0.001I_2,\quad Q_t = (5+5t)I_2,\quad x_0 = \mathrm{Concat}(0.2,0),\\
    &H\equiv \begin{bmatrix}2 & 1\\ 0 & 3\end{bmatrix},\quad \sigma_V = \begin{bmatrix}0.1 & 0.05\\ 0.05 & 0.1\end{bmatrix},\quad r_t= \begin{bmatrix}2t\\2t\end{bmatrix}.
\end{aligned}
\end{equation}
The underlying state dynamics admit a velocity-position interpretation, where the agent selects a velocity based on observations to track the reference trajectory \(r\). 

Figure~\ref{fig:2d_det} (resp. Figure~\ref{fig:2d_adpt}) depicts the trajectories of \(X^c\) and \(\hat{X}^a\) under the optimal deterministic (resp. adaptive) attacks \((\rho^*,\tau^*)\) for different values of \(\lambda\).
The qualitative behavior is consistent with Section~\ref{subsec:1d}. For visualization, trajectories are plotted in the state space with the time dimension suppressed, and arrows indicate the direction of motion.

\begin{figure}[htpb]
    \centering
    \includegraphics[width=0.9\linewidth]{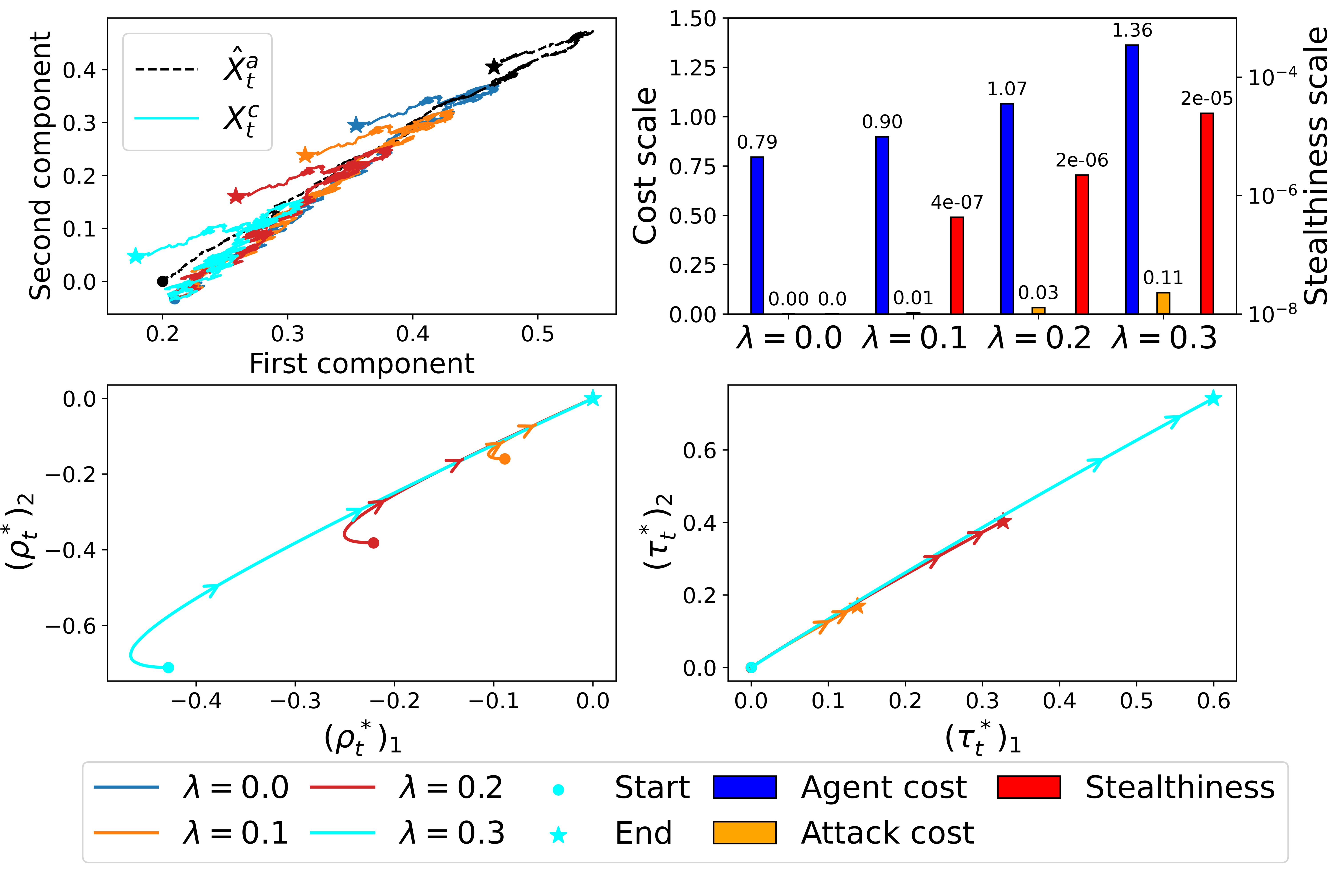}
    \caption{Trajectories under optimal deterministic attacks \((\rho^*,\tau^*)\) for different \(\lambda\) in the 2D tracking model. The state evolution is projected onto the spatial plane, with arrows indicating the direction of motion.
    The agent's filter \(\hat{X}^a\) is plotted only for \(\lambda = 0\) due to minimal variation across cases.
    }
    \label{fig:2d_det}
\end{figure}

\begin{figure}[htpb]
    \centering
    \includegraphics[width=0.9\linewidth]{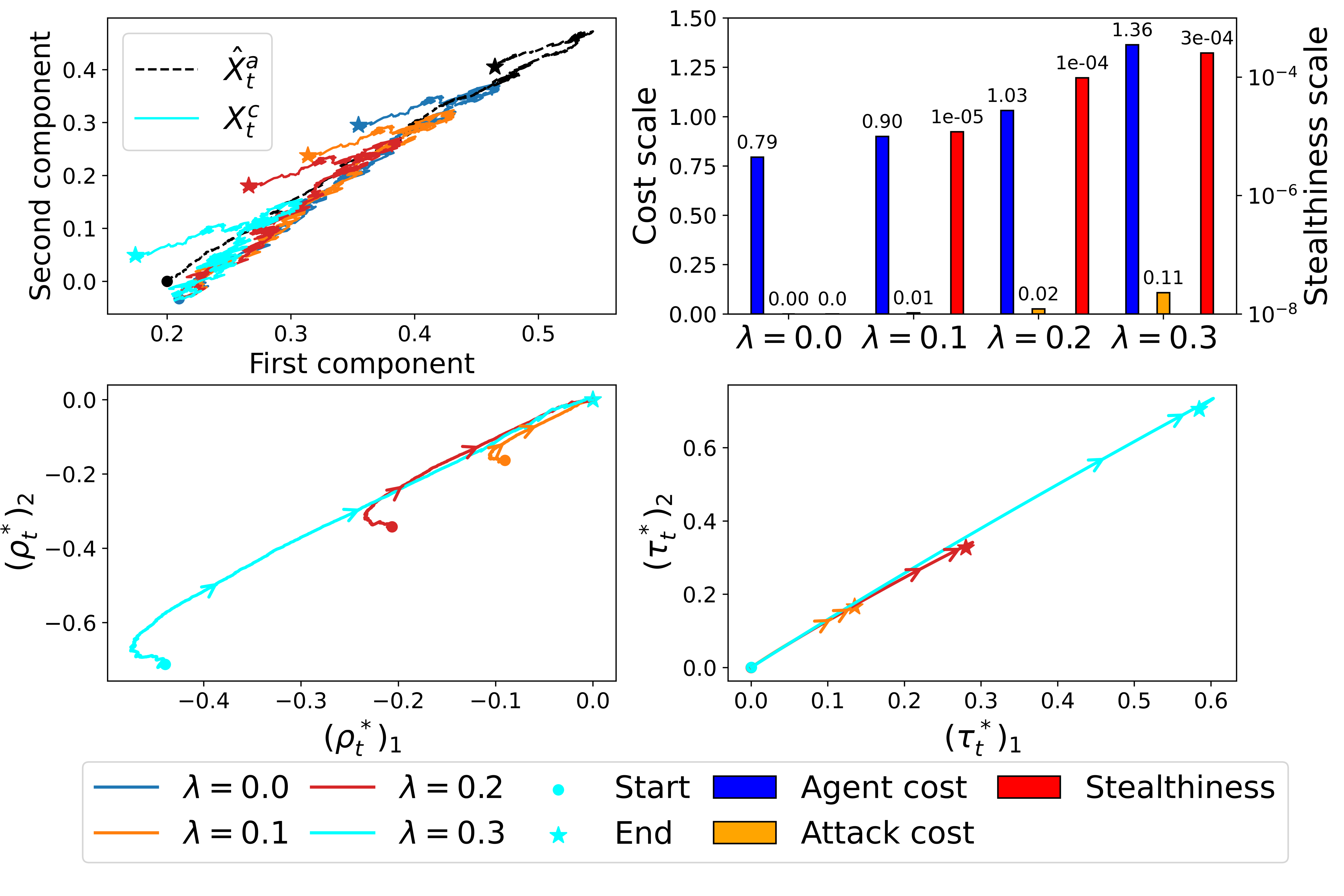}
    \caption{Trajectories under optimal adaptive attacks \((\rho^*,\tau^*)\) for different values of \(\lambda\) in the 2D tracking model. The state evolution is projected onto the spatial plane, with arrows indicating the direction of motion.
    The agent's filter \(\hat{X}^a\) is plotted only for \(\lambda = 0\) due to minimal variation across cases.
    }
    \label{fig:2d_adpt}
\end{figure}

\subsection{Multi-round extension}
As a final extension, we consider a multi-round interaction between an attacker adopting optimal deterministic attacks and an attack-aware defender with full knowledge about the injected attacks.
In each round, the defender updates its control \(u^*\) based on the previously applied attacks, treating \(a_t + \rho^*_t\) (resp. \(h_t + \tau^*_t\)) as the updated \(a_t\) (resp. \(h_t\)) in the next round. The attacker then recomputes the optimal deterministic attack in response, to further degrade the defender’s performance.

Figure 6 illustrates the evolution of the attack magnitude and the expected cost over multiple rounds, under the same model and parameters as in Section~\ref{subsec:1d}. As the interaction progresses, the defender partially mitigates the impact of attacks through adaptation, but the attacker counteracts by redesigning the attacks, leading to increasing attack magnitudes. This results in a monotone increase in the magnitude of the attacks and a corresponding rise in the system cost. These observations suggest that the interaction may fail to converge to a stable equilibrium under repeated adaptation. Moreover, even when the defender adjusts its strategy, the attacker remains effective in degrading performance. A rigorous analysis of this dynamic interaction is left for future work.

\begin{figure}[htbp]
    \centering
    \includegraphics[width=0.9\linewidth]{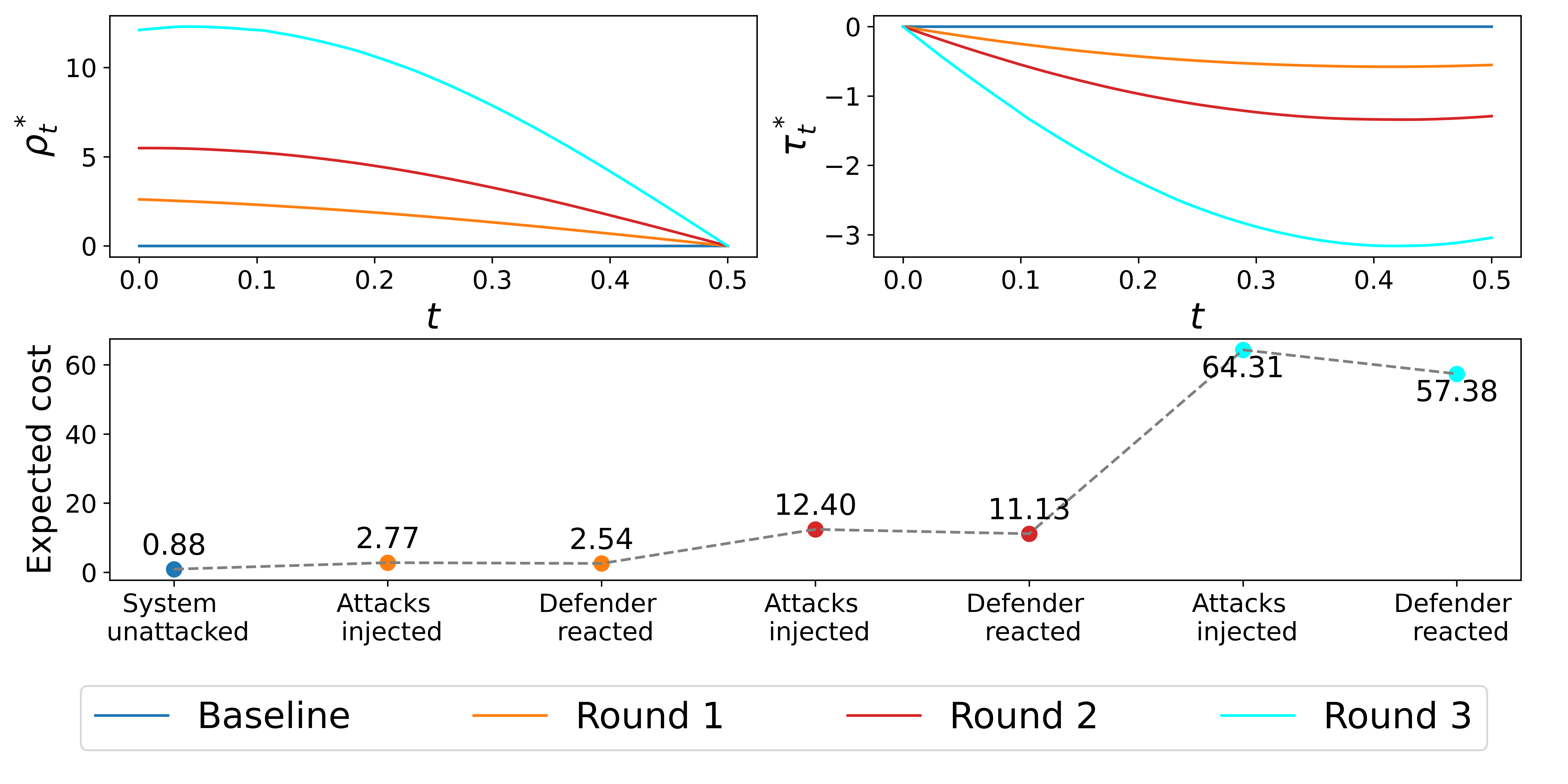}
    \caption{Evolution of attack magnitude and expected cost in the multi-round 1D mean-reverting model. At each round, the defender updates its control and the attacker recomputes the optimal deterministic attack with \(\lambda = 0.5\).
    }
    \label{fig:multi_round}
\end{figure}

Overall, the experiments demonstrate that the proposed framework produces highly effective yet statistically stealthy attacks under different information structures. Compared to heuristic ones, our attack designs achieve a better stealthiness-effectiveness trade-off.

\section{Conclusions and future studies}\label{sec:conclusions}

In this paper, we studied the optimal design of stealthy attacks against partially observable stochastic control systems.
Based on a novel likelihood-based attack detector constructed from the innovation process for linear filtering systems, we formulated deterministic and adaptive attack design problems from the attacker’s perspective.
By exploiting the linear-quadratic structure and applying control-theoretic techniques, we derived semi-explicit characterizations of the optimal attacks and established their well-posedness. Numerical experiments demonstrated the effectiveness and stealthiness of the proposed attacks compared with heuristic designs.

Several directions remain open for future research. One natural extension is the numerical study of stealthy attack design against GLRT-based detectors (see Remark~\ref{rem:interp_ell}).
Other direct extensions involve risk-sensitive agents \cite{whittle1981risk} or strategic attack-aware agents, who actively mislead the attacker by purposefully biasing its strategies \cite{zhou2025integrating,zhou2025adversarial}.
Beyond the present framework, further directions include nonlinear system dynamics where the separation principle no longer applies \cite{bensoussan1992stochastic}, stealthiness against alternative detection mechanisms such as sensor watermarking \cite{ferrari2017detection}, attacks on other system components \cite{fang2019stealthy}, and attacks against multi-agent networked systems \cite{babichenko2026forecasting,hu2025finite}.

\bibliographystyle{plain}        
\bibliography{ref}

\appendix 

\section{Proof of Proposition~\ref{prop:solution_agent} in Section~\ref{sec:agent}}

\begin{pf}[Proof of Proposition~\ref{prop:solution_agent}]
    By the separation principle \cite{davis1977}, solving the partially observable control problem~\eqref{eqn:X}--\eqref{eqn:adm_separation} is equivalent to solving the fully observable Markovian control problem under state dynamics~\eqref{eqn:X_hat}. 
    
    Denote by \(V(t,x)\) the associated value function, where \(x\in\R^d\) denotes the state variable for \(\hat{X}\).
    By the dynamic programming principle (DPP), \(V\) satisfies the Hamilton-Jacobi-Bellman (HJB) equation
\begin{multline}
    \partial_t V + \tfrac{1}{2}\mathrm{Tr}(R_tH_t\transpose (\sigma_W\sigma_W\transpose)^{-1}H_tR_t\transpose \partial_{xx} V) + \\
    \inf_u\big\{(A_tx + B_tu + a_t)\transpose \partial_{x} V \\
    + (x - r_t)\transpose Q_t (x - r_t) + u\transpose S_t u\big\} = 0,
\end{multline}
with terminal condition \(V(T,x) =0\).
Solving for the infimum yields
\begin{equation}
    u^*(t,x) = -\tfrac{1}{2}S_t^{-1}B_t\transpose \partial_x V.
\end{equation}
Plugging the optimal control \(u^*\) into the HJB equation, using the quadratic ansatz \(V(t,x) = x\transpose F_t x + \mb{f}_t\transpose x + c_t\), where \(c\in C([0,T];\R)\), and collecting coefficients yield the ODE system~\eqref{eqn:ODE_agent}.
Notably, the solution to \(c\) exists whenever \(F\) and \(\mb{f}\) exist.
Plugging the ansatz of the value function into \(u^*\) yields equation~\eqref{eqn:opt_ctrl}.

Finally, the well-posedness result follows from \cite[Theorem~4.1.6]{abou2012matrix}, which concludes the proof.
\end{pf}

\section{Proofs of Propositions~\ref{prop:detection}, \ref{prop:stealthy} and~\ref{prop:innov_BM} in Section~\ref{sec:detector} }\label{sec:A}
The proof of Proposition~\ref{prop:detection} relies on Lemma~\ref{lem:lkhd}, a variant of \cite[Theorem~7.13]{liptser2013statistics}.
Unlike the original result, this variant does not explicitly assume the martingale property of the exponential local martingale when applying the Girsanov theorem.
Instead, the Novikov condition is checked via a localization argument (see, e.g., the proof of \cite[Theorem~7.19]{liptser2013statistics}), yielding conditions that are easier to verify in our setting.

\begin{lem}
    \label{lem:lkhd}
    Let \(W\) be an \(\R^m\)-Brownian motion in  under the filtration \(\{\mathscr{G}_t\}\). 
    Let \(\xi\) be an \(\R^m\)-valued process satisfying \(\ud \xi_t = \beta_t\ud t + \ud W_t\) and \(\xi_0 = 0\), where \(\beta_t\in \mathscr{G}_t,\ \forall t\in[0,T]\).
    If the following conditions hold:
    \begin{enumerate}
        \item \(\int_0^T \E \|\beta_t\|\ud t<\infty\).
        \item \(\int_0^T \|\E (\beta_t|\F^\xi_t)\|^2\ud t <\infty \ \mathrm{a.s.}\).
    \end{enumerate}
    Then \(\mu_\xi \ll \mu_W\), and
    \begin{equation}
        \frac{\ud \mu_\xi}{\ud \mu_W}(\xi) = e^{\int_0^T \E(\beta_t|\F^\xi_t)\transpose\ud\xi_t - \frac{1}{2}\int_0^T \|\E(\beta_t|\F^\xi_t)\|^2\ud t}.
    \end{equation}
\end{lem}
\begin{pf}
    By \cite[Theorem~7.12]{liptser2013statistics}, \(\xi\) admits a diffusion-type representation \(\ud\xi_t = \E(\beta_t|\F^\xi_t)\ud t + \ud \BAR{W}_t\), where \(\BAR{W}\) is an \(\{\F^\xi_t\}\)-adapted Brownian motion.
    Since \(\mu_W = \mu_{\BAR{W}}\), applying \cite[Theorem~7.20]{liptser2013statistics} concludes the proof.
\end{pf}

\begin{pf}[Proof of Proposition~\ref{prop:detection}]
Recall the dynamics of \(I^a\) given by equation~\eqref{eqn:I^a_W}.
To apply Lemma~\ref{lem:lkhd} for \(\mathscr{G}_t = \F_t\), \(\xi_t = I^a_t\) and \(\beta_t = (\sigma_W\sigma_W\transpose)^{-\frac{1}{2}}[H_t(X^c_t - \hat{X}^a_t) + \tau_t]\), the two conditions need to be verified.
Here, we first claim that condition~\eqref{eqn:cond_rho_tau} implies \(\int_0^T \E \|\beta_t\|^2\ud t<\infty\).

To see why this claim holds, note that
\begin{equation}
    \int_0^T \E \|\beta_t\|^2\ud t \leq C \int_0^T \E (\|X^c_t - \hat{X}^a_t\|^2 + \|\tau_t\|^2)\ud t,
\end{equation}
where \(C\) denotes some positive absolute constant.
Under condition~\eqref{eqn:cond_rho_tau}, it suffices to prove \(\int_0^T \E \|X^c_t - \hat{X}^a_t\|^2\ud t<\infty\).
By the triangle inequality,
\begin{multline}
    \int_0^T \E \|X^c_t - \hat{X}^a_t\|^2\ud t \\
    \leq 2\int_0^T \E (\|X^c_t - \hat{X}^c_t\|^2 + \|\hat{X}^c_t - \hat{X}^a_t\|^2)\ud t.
\end{multline}
By standard estimates for equation~\eqref{eqn:Delta_X}, \(\int_0^T \E\|\rho_t - R_tH_t\transpose (\sigma_W\sigma_W\transpose)^{-1}\tau_t\|^2\ud t<\infty\) implies \(\int_0^T \E \|\Delta X_t\|^2\ud t<\infty\).
On the other hand, since \(R_t = \mathrm{cov}(X^c_t|\F^{Y^c}_t)\),
\begin{equation}
    \E \|X^c_t - \hat{X}^c_t\|^2 = \Tr[\E (X^c_t - \hat{X}^c_t)(X^c_t - \hat{X}^c_t)\transpose] = \Tr(R_t).
\end{equation}
As a result, \(\int_0^T \E \|X^c_t - \hat{X}^a_t\|^2\ud t<\infty\), which proves the desired claim.

We now proceed to verifying the two conditions in Lemma~\ref{lem:lkhd}.
Firstly, it is clear that \(\int_0^T \E \|\beta_t\|^2\ud t<\infty\) implies \(\int_0^T \E \|\beta_t\|\ud t<\infty\).
For the second condition, we prove a stronger result: \(\E\int_0^T \|\E (\beta_t|\F^\xi_t)\|^2\ud t <\infty\).
Indeed, by Jensen's inequality,
\begin{multline}
    \E\int_0^T \|\E (\beta_t|\F^\xi_t)\|^2\ud t \leq \E\int_0^T \E (\|\beta_t\|^2|\F^\xi_t)\ud t\\
    = \int_0^T \E \|\beta_t\|^2\ud t<\infty.
\end{multline}

Finally, since \(\F^{I^a}_t = \F^{Y^c}_t,\ \forall t\in[0,T]\), Proposition~\ref{prop:DeltaX} implies \(\E(\beta_t|\F^{I^a}_t) = (\sigma_W\sigma_W\transpose)^{-\frac{1}{2}}(H_t\Delta X_t + \tau_t)\).
Applying Lemma~\ref{lem:lkhd} concludes the proof.
\end{pf}

\begin{pf}[Proof of Proposition~\ref{prop:stealthy}]
    Plugging equation~\eqref{eqn:I^a_W} into~\eqref{eqn:log_lkhd} and taking expectation on both sides yield
    \begin{multline}
    \mc{S}(\rho,\tau) = \E \int_0^T (H_t\Delta X_t + \tau_t)\transpose (\sigma_W\sigma_W\transpose)^{-1}\\
    [H_t(X^c_t - \hat{X}^a_t) + \tau_t]\ud t \\
    - \frac{1}{2}\E \int_0^T (H_t\Delta X_t + \tau_t)\transpose (\sigma_W\sigma_W\transpose)^{-1}(H_t\Delta X_t + \tau_t)\ud t.
    \end{multline}
    By the tower property, 
    \begin{align}
        &\E \int_0^T (H_t\Delta X_t + \tau_t)\transpose (\sigma_W\sigma_W\transpose)^{-1}[H_t(X^c_t - \hat{X}^a_t) + \tau_t]\ud t\\
        &= \E \Big[\int_0^T (H_t\Delta X_t + \tau_t)\transpose (\sigma_W\sigma_W\transpose)^{-1}\\
        &\qquad \qquad \qquad \E \Big(H_t(X^c_t - \hat{X}^a_t) + \tau_t\Big| \F^{Y^c}_t\Big)\ud t \Big]\\
        &= \E \Big[\int_0^T (H_t\Delta X_t + \tau_t)\transpose (\sigma_W\sigma_W\transpose)^{-1}(H_t\Delta X_t + \tau_t)\ud t \Big],
    \end{align}
    which concludes the proof. 
\end{pf}

\begin{pf}[Proof of Proposition~\ref{prop:innov_BM}]
    For any \(0\leq s\leq t\leq T\), by equation~\eqref{eqn:I^a_W},
    \begin{align}
        & \E (I^a_t - I^a_s|\F^{Y^c}_s)\\
        &= (\sigma_W\sigma_W\transpose)^{-\frac{1}{2}}\E \Big( \int_s^t [H_u(X^c_u - \hat{X}^a_u) + \tau_u]\ud u \Big| \F^{Y^c}_s\Big) \\
        &\quad + (\sigma_W\sigma_W\transpose)^{-\frac{1}{2}}\sigma_W \E (W_t-W_s|\F^{Y^c}_s).
    \end{align}
    Since \(W_t-W_s\) is independent of \(\F_s\supset \F^{Y^c}_s\), by the tower property, \( \E (W_t-W_s|\F^{Y^c}_s) = 0\).
    Similarly, since \(\F^{Y^c}_s\subset \F^{Y^c}_u\), \(\E (X^c_u - \hat{X}^a_u| \F^{Y^c}_s) = \E (\Delta X_u| \F^{Y^c}_s)\).
    The equivalent condition for \(I^a\) being an \(\{\F^{Y^c}_t\}\)-adapted continuous martingale is thus proved.
    
    When \(\rho\) and \(\tau\) are deterministic, by Proposition~\ref{prop:DeltaX}, \(H\Delta X + \tau\) is a deterministic continuous function.
    The equivalent condition becomes \(H\Delta X + \tau \equiv 0\).
    Combining with equation~\eqref{eqn:Delta_X} yields
    \begin{equation}
        \ud \Delta X_t = (A_t\Delta X_t + \rho_t)\ud t,\quad \Delta X_0 = 0.
    \end{equation}
    When \(A_{t_1}\) and \(A_{t_2}\) commute \(\forall t_1,t_2\in[0,T]\), this ODE has an analytical solution 
    \begin{equation}
        \Delta X_t = e^{\int_0^t A_u\ud u}\int_0^t e^{-\int_0^s A_u\ud u} \rho_s\ud s,
    \end{equation}
    and \(H\Delta X + \tau \equiv 0\) is equivalent to condition~\eqref{eqn:innov_cond}.

    Finally, since \(I^a_0 = 0\) and the quadratic variation \(\QV{I^a}{I^a}_t = tI_m\) (cf. equation~\eqref{eqn:I^a_W}), by Levy's characterization, \(I^a\) is an \(\{\F^{Y^c}_t\}\)-adapted Brownian motion if and only if it is an \(\{\F^{Y^c}_t\}\)-adapted continuous martingale.
    This concludes the proof.
\end{pf}

\section{Proofs of Proposition~\ref{prop:degrad_det} and Theorem~\ref{thm:solution_det} in Section~\ref{sec:det}}\label{sec:C}

\begin{pf}[Proof of Proposition~\ref{prop:degrad_det}]
    Plugging equations~\eqref{eqn:Yc} and~\eqref{eqn:ua} into equation~\eqref{eqn:Xa} yields
    \begin{multline}
        \label{eqn:Xa_W}
        \ud \hat{X}^a_t = \Big[\Theta_tX^c_t + (A_t - K_t F_t - \Theta_t)\hat{X}^a_t + \mc{T}_t\tau_t \\
        + a_t - \tfrac12 K_t \mb{f}_t\Big]\ud t
        + \mc{T}_t\sigma_W\ud W_t,\quad \hat{X}^a_0 = x_0.
    \end{multline}
    Plugging equation~\eqref{eqn:ua} into equation~\eqref{eqn:Xc} yields
    \begin{multline}
        \ud X^c_t = \Big(A_tX^c_t -K_t F_t \hat{X}^a_t + \rho_t + a_t- \tfrac12 K_t \mb{f}_t \Big)\ud t + \sigma_V\ud V_t,\\ X^c_0\sim N(x_0,R_0).
    \end{multline}
    As a result, \(\Psi_t := \mathrm{Concat}(X_t^c,\hat{X}^a_t)\) satisfies the linear SDE:
    \begin{equation}
        \ud \Psi_t = \Big(\mathcal{A}_t\Psi_t + \mathcal{B}_t\begin{bmatrix}\rho_t\\\tau_t\end{bmatrix} + \alpha_t\Big)\ud t + \mathcal{V}_t\begin{bmatrix}\ud V_t\\\ud W_t\end{bmatrix},
    \end{equation}
    with a Gaussian initial condition \(\Psi_0\) and $\mathcal{A}_t$, $\mathcal{B}_t$, $\mathcal{V}_t$ specified in Proposition~\ref{prop:degrad_det}. This justifies the Gaussianity of \(\Psi_t\).
    
    Taking expectations on both sides of the dynamics of \(\Psi\) yields the dynamics of \(m\) (cf. equation~\eqref{eqn:m_Sigma}).
    Since \(\Sigma_t = \E (\Psi_t-m_t) (\Psi_t-m_t) \transpose\), applying It{\^o}'s formula yields 
    \begin{multline}
        \ud (\Psi_t-m_t) (\Psi_t-m_t) \transpose = (\Psi_t - m_t)\ud (\Psi_t-m_t)\transpose \\
        + [\ud(\Psi_t - m_t)] (\Psi_t - m_t)\transpose + \mathcal{V}_t\mathcal{V}_t\transpose \ud t.
    \end{multline}
    Using the dynamics of \(\Psi\) and \(m\), integrating and taking expectations on both sides yield the dynamics of \(\Sigma\) (cf. equation~\eqref{eqn:m_Sigma}).
    
    Lastly, we proceed to proving the expression of \(\mc{D}(\rho,\tau)\).
    Firstly, note that
    \begin{align}
        & \E (X^c_t - r_t)\transpose Q_t (X^c_t - r_t) \\
        &= \E (X^c_t - m^c_t)\transpose Q_t (X^c_t - m^c_t) \\
        &\quad +  (m^c_t - r_t)\transpose Q_t (m^c_t - r_t) \\
        &=\mathrm{Tr}(Q_t  \Sigma^{cc}_t) +  (m^c_t - r_t)\transpose Q_t (m^c_t - r_t).
    \end{align}
    Similarly, by equations~\eqref{eqn:opt_ctrl} and~\eqref{eqn:ua}, 
    \begin{align}
        \E (u^a_t)\transpose S_t u^a_t &= \E (F_t\hat{X}^a_t + \tfrac12 \mb{f}_t)\transpose K_t (F_t\hat{X}^a_t + \tfrac12 \mb{f}_t)\\
        &= \mathrm{Tr}(F_tK_t F_t\Sigma^{aa}_t) \\
        &\quad + (F_tm^a_t + \tfrac12 \mb{f}_t)\transpose K_t  (F_tm^a_t + \tfrac12 \mb{f}_t).
    \end{align}
    Plugging those equations into definition~\eqref{eqn:def_D} concludes the proof.
\end{pf}

In the proof of Theorem~\ref{thm:solution_det}, the construction of the existence interval is based on the following Lemma~\ref{lem:PSD}, which is a direct corollary of the block Gershgorin circle theorem \cite[Theorem~1.13.1]{tretter2008spectral}.

\begin{lem}
    \label{lem:PSD}
    For any block matrix \(A\in\mathbb{S}^{dn\times dn}\) consisting of \(n^2\) blocks \(A_{ij}\in \mathbb{R}^{d\times d},\ \forall i,j\in[n]\), if \(A_{ii}\geq 0\) and \(\lambda_{\min}(A_{ii})\geq \sum_{j\neq i}\|A_{ij}\|_2\) for any \( i\in[n]\), where \(\lambda_{\min}(\cdot)\) denotes the minimum eigenvalue of a symmetric matrix, then \(A\geq 0\).
    Here, matrix inequalities are understood in the positive semi-definite sense.
\end{lem}
    
    \begin{pf}[Proof of Theorem~\ref{thm:solution_det}]
        Following the Pontryagin's maximum principle \cite{mangasarian1966sufficient,seierstad1977sufficient}, we first specify the Hamiltonian associated with the control problem~\eqref{eqn:det_obj} under the state dynamics~\eqref{eqn:m_Sigma}.
        By Proposition~\ref{prop:degrad_det} and equation~\eqref{eqn:DeltaX_det}, we define the Hamiltonian $H$ by
    \begin{multline}
    \label{eqn:H}
    H(t,m^c,m^a,y^c,y^a,\rho,\tau) 
    := \\ (y^c)\transpose(A_tm^c -K_t F_t m^a + a_t - \tfrac12 K_t \mb{f}_t + \rho)+ \tfrac12\rho\transpose P_t\rho\\
     + (y^a)\transpose\Big[\Theta_tm^c +  (A_t - K_t F_t - \Theta_t)m^a + a_t - \tfrac12 K_t \mb{f}_t + \mc{T}_t\tau\Big]\\
    + \tfrac{1}{2}[H_t(m^c - m^a) + \tau]\transpose (\sigma_W\sigma_W\transpose)^{-1}[H_t(m^c - m^a) + \tau] \\
    - \lambda \Big[  (m^c - r_t)\transpose Q_t (m^c - r_t) \\
    + (F_tm^a + \tfrac12 \mb{f}_t)\transpose K_t  (F_tm^a + \tfrac12 \mb{f}_t)\Big],
\end{multline}
where \(y^c,y^a\in\R^d\) denote the adjoint variables associated with \(m^c,m^a\).

Minimizing the Hamiltonian with respect to \(\rho\) and \(\tau\) yields the candidate of optimal controls
\begin{equation}
    \rho^* = -P_t^{-1}y^c,\quad \tau^* = -H_tR_ty^a - H_t(m^c - m^a).
\end{equation}
Plugging the optimal controls \(\rho^*\) and \(\tau^*\) into equation~\eqref{eqn:m_Sigma} yields the forward equations
\begin{equation}
\begin{aligned}
    \dot{m}^c_t &= A_tm^c_t -K_t F_t m^a_t + a_t - \tfrac12 K_t \mb{f}_t -P_t^{-1}y^c_t,\\
    \dot{m}^a_t &=   (A_t - K_t F_t)m^a_t + a_t - \tfrac12 K_t \mb{f}_t - \Lambda_ty^a_t,
\end{aligned}
\end{equation}
with \(m^c_0 = m^a_0 = x_0\).
Next, we compute the state gradients of the Hamiltonian
\begin{align}
    \partial_{m^c}H &= A_t\transpose y^c + \Theta_t\transpose y^a + H_t\transpose(\sigma_W\sigma_W\transpose)^{-1}\\
    &\quad [H_t(m^c-m^a)+\tau]  - 2\lambda Q_t(m^c - r_t) ,\\
    \partial_{m^a}H &= -F_tK_t y^c + (A_t\transpose - F_tK_t - \Theta_t\transpose)y^a  \\
    &\quad - H_t\transpose(\sigma_W\sigma_W\transpose)^{-1}[H_t(m^c-m^a) + \tau]\\
    &\qquad - \lambda F_tK_t (2F_tm^a+\mb{f}_t),
\end{align}
which provide the adjoint equations satisfied by \(y^c\) and \(y^a\):
\begin{equation}
\begin{aligned}
    \dot{y}^c_t &= -A_t\transpose y^c_t + 2\lambda Q_t(m^c_t - r_t),\\
    \dot{y}^a_t &= F_tK_t y^c_t - (A_t\transpose - F_tK_t)y^a_t + \lambda F_tK_t (2F_tm^a_t+\mb{f}_t),
\end{aligned}
\end{equation}
with \(y^a_T = y^c_T = 0\).

The optimal controls \(\rho^*\) and \(\tau^*\) are thus characterized by the solution to the forward-backward-ODE (FBODE) system above.
We solve this system by using the affine ansatz
\begin{equation}
    y^c_t = F^c_tm^c_t + F^a_tm^a_t + \mb{f}^\rho_t,\quad y^a_t = G^c_tm^c_t + G^a_tm^a_t + \mb{g}^\tau_t.
\end{equation}
Plugging the ansatz into the FBODEs and collecting coefficients yield system~\eqref{eqn:ODE_det}, which proves the first part of Theorem~\ref{thm:solution_det}.

Next, we prove well-posedness results associated with the ODE system~\eqref{eqn:ODE_det}.
We first focus on establishing existence intervals for \(F^c,F^a,G^c,G^a\).
    Let
    \begin{equation}
    M_t := 
    \begin{bmatrix}
        F^c_t & F^a_t\\
        G^c_t & G^a_t
    \end{bmatrix}\in\R^{2d\times 2d}.
    \end{equation}
    Rewriting the ODE system~\eqref{eqn:ODE_det} in terms of \(M\) yields
    \begin{equation}
    \dot{M}_t + \mc{P}_t\transpose M_t + M_t \mc{P}_t - M_t\mc{R}_tM_t + \mc{Q}_t = 0,\quad M_T = 0,
\end{equation}
where the coefficients
\begin{multline}
    \mc{P}_t := \begin{bmatrix}
        A_t & -K_tF_t\\
        0 & A_t-K_tF_t
    \end{bmatrix},\\
    \mc{Q}_t := \begin{bmatrix}
        -2\lambda Q_t & 0\\
        0 & -2\lambda F_tK_tF_t
    \end{bmatrix},\ 
    \mc{R}_t := \begin{bmatrix}
        P^{-1}_t & 0\\
        0 & \Lambda_t
    \end{bmatrix}.
\end{multline}
    Since \(\mc{R}_t,M_T\geq 0\) and \(\mc{Q}_t\) are all symmetric, \(M_t\) is also symmetric.
    When \(\lambda = 0\), \(\mc{Q}_t\geq 0\), and the global well-posedness follows directly from \cite[Theorem~4.1.6]{abou2012matrix}.
    
    For strictly positive \(\lambda\), the negativity of \(\mc{Q}_t\) becomes the main obstruction to the global well-posedness of \(M\).
    We follow an approach similar to that of \cite[Theorem~3.4]{hu2025strategic} to establish the existence interval.

    \smallskip
\noindent    \textbf{Step 1. Eliminate linear terms in the ODE for \(M\).}
    Since \( \mc{Q}_t\geq -2\lambda(\|Q\|\vee \|FKF\|)I_{2d}\) and \(\mc{R}_t\leq (\|P^{-1}\|\vee \|\Lambda\|)I_{2d}\), by Lemma~\ref{lem:PSD},
\begin{multline}
    {\scriptsize
    \begin{bmatrix}
        -[2\lambda(\|Q\|\vee \|FKF\|)+ \|\mc{P}\|]I_{2d} & 0\\
        0 & -(\|P^{-1}\|\vee \|\Lambda\| + \|\mc{P}\|)I_{2d}
    \end{bmatrix}}
    \\ \leq
    \begin{bmatrix}
        \mc{Q}_t & \mc{P}_t\transpose\\
        \mc{P}_t & -\mc{R}_t
    \end{bmatrix}.
\end{multline}
    By the comparison principle of Riccati equations \cite[Theorem~4.1.4]{abou2012matrix}, \(M_t\geq L_t\),
where \(L:[0,T]\to \R^{2d\times 2d}\) solves a Riccati ODE without linear terms
\begin{multline}
    \dot{L}_t =[2\lambda(\|Q\|\vee \|FKF\|)+ \|\mc{P}\|]I_{2d} \\
    + L_t(\|P^{-1}\|\vee \|\Lambda\| + \|\mc{P}\|)L_t,\ L_T = 0.
\end{multline}

\noindent    \textbf{Step 2. Scalarize the ODE for the lower bound \(L\).}
    Using the ansatz \(L_t = l_t I_{2d}\), where \(l:[0,T]\to\R\), the ODE for \(L\) reduces to a scalar-valued equation
\begin{equation}
    \dot{l}_t =[2\lambda(\|Q\|\vee \|FKF\|)+ \|\mc{P}\|] + (\|P^{-1}\|\vee \|\Lambda\| + \|\mc{P}\|)l_t^2,
\end{equation}
with \(l_T = 0\),
which has an analytical solution
\begin{equation}
    l_t = \sqrt{p/q}\tan(\sqrt{pq}(t-T)),
\end{equation}
where \(p := 2\lambda(\|Q\|\vee \|FKF\|)+ \|\mc{P}\|\), \(q:=\|P^{-1}\|\vee \|\Lambda\| + \|\mc{P}\|\) and \(\|\mc{P}\|\leq b_{\mc{P}}\).
Clearly, \(l_t\) exists on the time interval \((T-\frac{\pi}{2\sqrt{pq}},T]\), i.e., when \(T<\frac{\pi}{2\sqrt{pq}}\), \(l_t\) exists on the entire time horizon \([0,T]\).

\smallskip
\noindent\textbf{Step 3. Establish a global upper bound.}
By the comparison principle \cite[Theorem~4.1.4]{abou2012matrix}, \(M_t\leq N_t\), where \(N:[0,T]\to \R^{2d\times 2d}\) solves a linear ODE \(\dot{N}_t + \mc{P}_t\transpose N_t + N_t \mc{P}_t + \mc{Q}_t = 0,\ N_T = 0\), whose global well-posedness is trivial to obtain.

Combining the conclusions from all steps above, the \textit{a priori} lower and upper bounds of \(M\) exist when 
\begin{equation}
    T < \frac{\pi}{2\sqrt{[2\lambda(\|Q\|\vee \|FKF\|)+ b_{\mc{P}}](\|P^{-1}\|\vee \|\Lambda\| + b_{\mc{P}})}}.
\end{equation}
By the Picard-Lindel\"of theorem \cite{teschl2012ordinary}, \(M\) uniquely exists on \([0,T]\).

Given the existence of \(M\), \(\mb{f}^\rho\) and \(\mb{g}^\tau\) satisfy linear ODEs, whose well-posedness is trivial to obtain. This concludes the proof.
\end{pf}

\section{Proofs of Theorems~\ref{thm:opt_adap_rho} and~\ref{thm:opt_adap_tau} in Section~\ref{sec:adpt}}\label{sec:D}

\begin{pf}[Proof of Theorem~\ref{thm:opt_adap_rho}]
    The proof consists of the following two steps.

    \medskip
  \noindent  \textbf{Step 1. Apply the separation principle.}
    Recall the dynamics~\eqref{eqn:Xc}--\eqref{eqn:Delta_X}.
    For any \(( \rho,\tau)\in\mathscr{A}^{\mathrm{adap}}\), \(\hat{X}^a\) and \(\Delta X\) are both \(\{\F^{Y^c}_t\}\)-adapted so that 
    \begin{equation}
        \E(X^c_t|\F^{Y^c}_t,\F^{\hat{X}^a}_t,\F^{\Delta X}_t) = \E(X^c_t|\F^{Y^c}_t) = \hat{X}^c_t.
    \end{equation}
    Consider a fictitious partially observable control problem, where \(X^c\) is the underlying state, and \(Y^c,\hat{X}^a,\Delta X\) are observations.
    We observe that (i) all the dynamics are linear, (ii) in equation~\eqref{eqn:X_hat_c}, the term \(B_tu^a_t + \rho_t\in \F^{Y^c}_t\), and most importantly, (iii) \(\tau\) is fixed in \textbf{Step (a)}, so no optimization with respect to \(\tau\) is involved.
    Consequently, the separation principle \cite{davis1977} applies.
    The partially observable control problem is equivalent to a Markovian control problem with complete observability of states \(\hat{X}^c,Y^c,\hat{X}^a,\Delta X\).

Denote by \(I^c\) the whitened innovation process associated with equation~\eqref{eqn:X_hat_c}, defined by
\begin{equation}
    \ud I^c_t := (\sigma_W\sigma_W\transpose)^{-\frac{1}{2}}[\ud Y^c_t - (H_t\hat{X}^c_t + h_t + \tau_t)\ud t],\quad I^c_0 = 0,
\end{equation}
which is an \(\{\F^{Y^c}_t\}\)-adapted Brownian motion independent of \(\rho\).
The dynamics of \(Y^c,\hat{X}^c,\hat{X}^a\) can thus be rewritten as driven by the innovation Brownian motion \(I^c\):
\begin{align}
    &\ud Y^c_t = (H_t\hat{X}^c_t + h_t + \tau_t)\ud t + (\sigma_W\sigma_W\transpose)^{\frac12}\ud I^c_t,\\
    &\ud \hat{X}^c_t = (A_t \hat{X}^c_t - K_tF_t\hat{X}^a_t + a_t - \tfrac12 K_t\mb{f}_t + \rho_t)\ud t \\
    &\qquad \qquad+ R_tH_t\transpose (\sigma_W\sigma_W\transpose)^{-\frac12}\ud I^c_t,\\
    & \ud \hat{X}^a_t = [\Theta_t\hat{X}^c_t +(A_t - K_t F_t - \Theta_t)\hat{X}^a_t + \mc{T}_t\tau_t \\
    &\qquad \qquad+ a_t - \tfrac12 K_t \mb{f}_t]\ud t + R_tH_t\transpose (\sigma_W\sigma_W\transpose)^{-\frac12}\ud I^c_t,
\end{align}
with given initial conditions \(Y^c_0 = 0\) and \(\hat{X}^c_0 = \hat{X}^a_0 = x_0\).

Next, we eliminate the unobserved state \(X^c\) from the running cost of the control problem~\eqref{eqn:adap_obj}. 
Since \(R_t = \mathrm{cov}(X^c_t|\F^{Y^c}_t)\), by the tower property,
\begin{align}
    &\E (X^c_t - r_t)\transpose Q_t (X^c_t - r_t)\ud t \\
    &= \E  (X^c_t - \hat{X}^c_t + \hat{X}^c_t - r_t)\transpose Q_t (X^c_t - \hat{X}^c_t + \hat{X}^c_t - r_t)\\
    &= \E  (\hat{X}^c_t - r_t)\transpose Q_t (\hat{X}^c_t - r_t) \\
    &\quad +  \E  (X^c_t - \hat{X}^c_t)\transpose Q_t (X^c_t - \hat{X}^c_t)\\
    &= \E  (\hat{X}^c_t - r_t)\transpose Q_t (\hat{X}^c_t - r_t) +\mathrm{Tr}(Q_tR_t),
\end{align}
where the crossing terms vanish due to the tower property that
\begin{align}
    &\E(X^c_t - \hat{X}^c_t)\transpose Q_t(\hat{X}^c_t - r_t)\\
    &= \E \Big[\E\Big((X^c_t - \hat{X}^c_t)\transpose Q_t(\hat{X}^c_t - r_t)\Big|\F^{Y^c}_t\Big)\Big]\\
    &= \E \Big[\E\Big(X^c_t - \hat{X}^c_t\Big|\F^{Y^c}_t\Big)\transpose Q_t(\hat{X}^c_t - r_t)\Big] = 0.
\end{align}
Since the trace term is independent of \(\rho\) and \(\tau\), it contributes only an additive constant to the objective and is therefore omitted throughout the optimization procedure.

So far, problem~\eqref{eqn:adap_obj} has been reduced to solving a Markovian control problem with control \(\rho\) and state processes \(Y^c,\hat{X}^c,\hat{X}^a,\Delta X\) driven by the innovation Brownian motion \(I^c\) under \(\{\F^{Y^c}_t\}\).

\medskip
\noindent\textbf{Step 2. Solve the reduced Markovian control problem.}
Note that \(Y^c\) does not explicitly enter the cost functional of the control problem.
Moreover, after rewriting the filtering dynamics in terms of the innovation process \(I^c\), the state dynamics of \(\hat{X}^c,\hat{X}^a,\Delta X\) are closed and no longer contains \(Y^c\) explicitly. Therefore, \(Y^c\) does not need to be included as an independent state variable in the subsequent control problem; its effect is captured through the innovation noise $I^c$, and the filtered state variables $(\hat X^c, \hat X^a)$ and their difference.

Denote by \(V(t,x^c,x^a,\Delta x)\) the value function of the Markovian control problem, where \(x^c,x^a,\Delta x\in\R^d\) are state variables associated with \(\hat{X}^c,\hat{X}^a,\Delta X\), respectively.
By the DPP, \(V\) satisfies the HJB equation
\begin{multline}
    \partial_t V + \tfrac{1}{2}\mathrm{Tr}(\Lambda_t\partial_{x^cx^c}V) + \tfrac{1}{2}\mathrm{Tr}(\Lambda_t\partial_{x^ax^a}V) + \mathrm{Tr}(\Lambda_t\partial_{x^ax^c}V) \\
    + \inf_\rho\Big\{(\partial_{x^c} V)\transpose(A_t x^c - K_tF_tx^a + a_t - \tfrac12 K_t\mb{f}_t + \rho)
    \\
    +(\partial_{x^a}V)\transpose\big[\Theta_tx^c +(A_t - K_t F_t - \Theta_t)x^a + \mc{T}_t\tau_t + a_t - \tfrac12 K_t \mb{f}_t\big]\\
    +(\partial_{\Delta x}V)\transpose\big[(A_t - \Theta_t)\Delta x + \rho - \mc{T}_t\tau_t\big]\\
    + \tfrac12(H_t\Delta x + \tau_t)\transpose (\sigma_W\sigma_W\transpose)^{-1}(H_t\Delta x + \tau_t) + \tfrac12\rho\transpose P_t\rho \\
    - \lambda  \big[(x^c - r_t)\transpose Q_t (x^c - r_t) \\
    + (F_t x^a + \tfrac12  \mb{f}_t)\transpose K_t (F_t x^a + \tfrac12 \mb{f}_t)\big]\Big\} = 0,
\end{multline}
with terminal condition \(V(T,x^c,x^a,\Delta x) = 0\).
Solving the infimum yields
\begin{equation}
    \rho^*(t,x^c,x^a,\Delta x) = -P_t^{-1}(\partial_{x^c} V + \partial_{\Delta x}V).
\end{equation}
Plug \(\rho^*\) back into the HJB equation and rewrite the equation in the vectorized form in terms of the augmented state variable \(\phi:= \mathrm{Concat}(x^c,x^a,\Delta x) \in \R^{3d}\).
By using the quadratic ansatz 
\begin{equation}
    V(t,\phi) = \tfrac12 \phi\transpose F^\phi_t\phi + \phi\transpose\mb{f}^\phi_t + c^\phi_t,
\end{equation}
where \(c^\phi\in C([0,T];\R)\), and collecting coefficients, we obtain the ODE system~\eqref{eqn:ODE_adap_rho}.

Notably, \(c^\phi\) satisfies the equation
\begin{multline}
    \dot{c}^\phi_t + \tfrac12 \mathrm{Tr}(\Sigma^\phi_tF^\phi_t) - \tfrac12 (\mb{f}^\phi_t)\transpose O_t\mb{f}^\phi_t + (\mb{f}^\phi_t)\transpose (d^\phi_t + d^\tau_t) \\
    +  C^\phi_t + C^\tau_t = 0, \quad c^\phi_T = 0,
\end{multline}
where \(\Sigma^\phi\) and \(C^\phi\) remain \(\tau\)-independent, whereas \(C^\tau\) depends on \(\tau\):
\begin{align}
    &\Sigma^\phi_t := \begin{bmatrix}
        \Lambda_t & \Lambda_t & 0\\
        \Lambda_t & \Lambda_t & 0\\
        0 & 0 & 0
    \end{bmatrix},\quad C^\phi_t := -\lambda(r_t\transpose Q_tr_t + \tfrac14 \mb{f}_t\transpose K_t \mb{f}_t), \\ &C^\tau_t := \tfrac12 \tau_t\transpose(\sigma_W\sigma_W\transpose)^{-1}\tau_t.
\end{align}
Clearly, this ODE for \(c^\phi\) admits a unique global solution given \(F^\phi\) and \(\mb{f}^\phi\).

Finally, the proof of the well-posedness result is similar to that of Theorem~\ref{thm:solution_det} and is thus omitted.
This concludes the proof.
\end{pf}

\begin{pf}[Proof of Theorem~\ref{thm:opt_adap_tau}]
    The proof consists of the following two steps.

    \medskip
    \noindent\textbf{Step 1. Formulate the optimization with respect to \(\tau\).}
    By the definition of the value function \(V\) in the proof of Theorem~\ref{thm:opt_adap_rho}, the expected cost under the attack tuple \((\rho^*(\tau),\tau)\) is 
    \begin{equation}
        \E[V(0,\Phi_0)] =  \tfrac12 \Phi_0\transpose F^\phi_0\Phi_0 + \Phi_0\transpose\mb{f}^\phi_0 + c^\phi_0,
    \end{equation}
where \(\Phi_0 = \mathrm{Concat}(x_0,x_0,0)\in\R^{3d}\) is deterministic.
Importantly, \(\mb{f}^\phi\) and \(c^\phi\) depend on \(\tau\), while \(F^\phi\) remains \(\tau\)-independent.
Hence, it suffices to minimize \(\Phi_0\transpose\mb{f}^\phi_0 + c^\phi_0\) with respect to \(\tau\).

Using the ODEs for \(\mb{f}^\phi\) (cf. equation~\eqref{eqn:ODE_adap_rho}) and \(c^\phi\) (cf. proof of Theorem~\ref{thm:opt_adap_rho}), we get
\begin{multline}
    \Phi_0\transpose\mb{f}^\phi_0 + c^\phi_0 = \int_0^T \Phi_0\transpose [- F^\phi_tO_t\mb{f}^\phi_t \\
    + F^\phi_t(d^\phi_t + d^\tau_t) + (D^\phi_t)\transpose\mb{f}^\phi_t + \ell^\phi_t + \ell^\tau_t]\ud t\\
    + \int_0^T \tfrac12 \mathrm{Tr}(\Sigma^\phi_t F^\phi_t) - \tfrac12 (\mb{f}^\phi_t)\transpose O_t\mb{f}^\phi_t + (\mb{f}^\phi_t)\transpose (d^\phi_t + d^\tau_t) \\
    +  C^\phi_t + C^\tau_t \ud t.
\end{multline}
Removing all \(\tau\)-independent parts yields the optimization objective
\begin{multline}
    J(\tau) := \int_0^T \Phi_0\transpose[(D^\phi_t)\transpose- F^\phi_tO_t]\mb{f}^\phi_t + \Phi_0\transpose F^\phi_t d^\tau_t \\
    + \Phi_0\transpose\ell^\tau_t  - \tfrac12 (\mb{f}^\phi_t)\transpose O_t\mb{f}^\phi_t
    + (d^\phi_t + d^\tau_t)\transpose\mb{f}^\phi_t   + C^\tau_t \ud t.
\end{multline}
The definitions of the coefficients (cf. Theorem~\ref{thm:opt_adap_rho}) further suggest that the integrand of \(J(\tau)\) only contains linear and quadratic functions in \(\mb{f}^\phi\) and \(\tau\).
Consequently, the optimization in \(\tau\) can be viewed as an LQ optimal control problem with state \(\mb{f}^\phi\) and open-loop control \(\tau\).

\medskip
\noindent\textbf{Step 2. Solve the optimal control problem.}
Since the state process \(\mb{f}^\phi\) evolves backward in time (cf. equation~\eqref{eqn:ODE_adap_rho}), we first conduct a time reflection
\begin{equation}
    U_t := \mb{f}^\phi_{T-t}, \quad \eta_t := \tau_{T-t},
\end{equation}
yielding the following dynamics of \(U\):
\begin{multline}
    \dot{U}_t = [(D^\phi_{T-t})\transpose- F^\phi_{T-t}O_{T-t}]U_t + F^\phi_{T-t}d^\phi_{T-t} \\
    + \ell^\phi_{T-t} + Q^F_{T-t}(\sigma_W\sigma_W\transpose)^{-1}\eta_t, \quad U_0 = 0.
\end{multline}
Similarly, \(J(\tau)\) can be represented in terms of \(U\) as follows
\begin{multline}
    J(\tau) = \int_0^T [\Phi_0\transpose((D^\phi_{T-t})\transpose- F^\phi_{T-t}O_{T-t}) + (d^\phi_{T-t})\transpose]U_t \\
    - \frac12 U_t\transpose O_{T-t}U_t + 
    U_t\transpose G_{T-t}\eta_t  
    \\ + \Phi_0\transpose Q^F_{T-t}(\sigma_W\sigma_W\transpose)^{-1}\eta_t  + \frac12 \eta_t\transpose(\sigma_W\sigma_W\transpose)^{-1}\eta_t \ud t.
\end{multline}

By Pontryagin's maximum principle \cite{mangasarian1966sufficient,seierstad1977sufficient}, the Hamiltonian is given by
\begin{multline}
    H(t,u,y,\eta) := y\transpose \Big([(D^\phi_{T-t})\transpose- F^\phi_{T-t}O_{T-t}]u \\
    + F^\phi_{T-t}d^\phi_{T-t}
    + \ell^\phi_{T-t} 
    + Q^F_{T-t}(\sigma_W\sigma_W\transpose)^{-1}\eta\Big)\\
    + \Big([\Phi_0\transpose((D^\phi_{T-t})\transpose- F^\phi_{T-t}O_{T-t}) \\
    + (d^\phi_{T-t})\transpose]u
    - \tfrac12 u\transpose O_{T-t}u +
    u\transpose G_{T-t}\eta 
     \\+ \Phi_0\transpose Q^F_{T-t}(\sigma_W\sigma_W\transpose)^{-1}\eta
     + \tfrac12 \eta\transpose(\sigma_W\sigma_W\transpose)^{-1}\eta\Big),
\end{multline}
where \(y\in \R^{3d}\) denotes the adjoint variable associated with \(u\).

Minimizing the Hamiltonian with respect to \(\eta\) yields
\begin{equation}
    \label{eqn:opt_eta_y}
    \eta^* = -(\sigma_W\sigma_W\transpose)\Big([Q^F_{T-t}(\sigma_W\sigma_W\transpose)^{-1}]\transpose(y + \Phi_0) + G_{T-t}\transpose u\Big).
\end{equation}
Calculate the state gradient of the Hamiltonian
\begin{multline}
    \partial_u H = (D^\phi_{T-t} - O_{T-t}\transpose (F^\phi_{T-t})\transpose)(y + \Phi_0) + d^\phi_{T-t} \\
    - O_{T-t}u + G_{T-t}\eta,
\end{multline}
which yields the adjoint equation
\begin{multline}
    \dot{y}_t = [G_{T-t}(Q^F_{T-t})\transpose - D^\phi_{T-t} + O_{T-t}\transpose(F^\phi_{T-t})\transpose] (y_t + \Phi_0)\\
    - d^\phi_{T-t}
    + (O_{T-t} + G_{T-t}\sigma_W\sigma_W\transpose G_{T-t}\transpose)U_t,\quad y_T = 0.
\end{multline}

Plugging the affine ansatz 
\begin{equation}
    y_t = F^\tau_{T-t} U_t + \mb{f}^\tau_{T-t}
\end{equation}
into the adjoint equation and collecting coefficients yield the ODE system~\eqref{eqn:ODE_U}, which concludes the proof.
\end{pf}

\end{document}